\documentclass[reqno]{amsart}

\usepackage{etoolbox}

\makeatletter
\let\old@tocline\@tocline
\let\section@tocline\@tocline

\newcommand{\subsection@dotsep}{4.5}
\newcommand{\subsubsection@dotsep}{4.5}
\patchcmd{\@tocline}
  {\hfil}
  {\nobreak
     \leaders\hbox{$\m@th
        \mkern \subsection@dotsep mu\hbox{.}\mkern \subsection@dotsep mu$}\hfill
     \nobreak}{}{}
\let\subsection@tocline\@tocline
\let\@tocline\old@tocline

\patchcmd{\@tocline}
  {\hfil}
  {\nobreak
     \leaders\hbox{$\m@th
        \mkern \subsubsection@dotsep mu\hbox{.}\mkern \subsubsection@dotsep mu$}\hfill
     \nobreak}{}{}
\let\subsubsection@tocline\@tocline
\let\@tocline\old@tocline

\let\old@l@subsection\l@subsection
\let\old@l@subsubsection\l@subsubsection

\def\@tocwriteb#1#2#3{%
  \begingroup
    \@xp\def\csname #2@tocline\endcsname##1##2##3##4##5##6{%
      \ifnum##1>\c@tocdepth
      \else \sbox\z@{##5\let\indentlabel\@tochangmeasure##6}\fi}%
    \csname l@#2\endcsname{#1{\csname#2name\endcsname}{\@secnumber}{}}%
  \endgroup
  \addcontentsline{toc}{#2}%
    {\protect#1{\csname#2name\endcsname}{\@secnumber}{#3}}}%

\newlength{\@tocsectionindent}
\newlength{\@tocsubsectionindent}
\newlength{\@tocsubsubsectionindent}
\newlength{\@tocsectionnumwidth}
\newlength{\@tocsubsectionnumwidth}
\newlength{\@tocsubsubsectionnumwidth}
\newcommand{\settocsectionnumwidth}[1]{\setlength{\@tocsectionnumwidth}{#1}}
\newcommand{\settocsubsectionnumwidth}[1]{\setlength{\@tocsubsectionnumwidth}{#1}}
\newcommand{\settocsubsubsectionnumwidth}[1]{\setlength{\@tocsubsubsectionnumwidth}{#1}}
\newcommand{\settocsectionindent}[1]{\setlength{\@tocsectionindent}{#1}}
\newcommand{\settocsubsectionindent}[1]{\setlength{\@tocsubsectionindent}{#1}}
\newcommand{\settocsubsubsectionindent}[1]{\setlength{\@tocsubsubsectionindent}{#1}}

\renewcommand{\l@section}{\section@tocline{1}{\@tocsectionvskip}{\@tocsectionindent}{\@tocsectionnumwidth}{\@tocsectionformat}}%
\renewcommand{\l@subsection}{\subsection@tocline{1}{\@tocsubsectionvskip}{\@tocsubsectionindent}{\@tocsubsectionnumwidth}{\@tocsubsectionformat}}%
\renewcommand{\l@subsubsection}{\subsubsection@tocline{1}{\@tocsubsubsectionvskip}{\@tocsubsubsectionindent}{\@tocsubsubsectionnumwidth}{\@tocsubsubsectionformat}}%
\newcommand{\@tocsectionformat}{}
\newcommand{\@tocsubsectionformat}{}
\newcommand{\@tocsubsubsectionformat}{}
\expandafter\def\csname toc@1format\endcsname{\@tocsectionformat}
\expandafter\def\csname toc@2format\endcsname{\@tocsubsectionformat}
\expandafter\def\csname toc@3format\endcsname{\@tocsubsubsectionformat}
\newcommand{\settocsectionformat}[1]{\renewcommand{\@tocsectionformat}{#1}}
\newcommand{\settocsubsectionformat}[1]{\renewcommand{\@tocsubsectionformat}{#1}}
\newcommand{\settocsubsubsectionformat}[1]{\renewcommand{\@tocsubsubsectionformat}{#1}}
\newlength{\@tocsectionvskip}
\newcommand{\settocsectionvskip}[1]{\setlength{\@tocsectionvskip}{#1}}
\newlength{\@tocsubsectionvskip}
\newcommand{\settocsubsectionvskip}[1]{\setlength{\@tocsubsectionvskip}{#1}}
\newlength{\@tocsubsubsectionvskip}
\newcommand{\settocsubsubsectionvskip}[1]{\setlength{\@tocsubsubsectionvskip}{#1}}

\patchcmd{\tocsection}{\indentlabel}{\makebox[\@tocsectionnumwidth][l]}{}{}
\patchcmd{\tocsubsection}{\indentlabel}{\makebox[\@tocsubsectionnumwidth][l]}{}{}
\patchcmd{\tocsubsubsection}{\indentlabel}{\makebox[\@tocsubsubsectionnumwidth][l]}{}{}

\newcommand{\@sectypepnumformat}{}
\renewcommand{\contentsline}[1]{%
  \expandafter\let\expandafter\@sectypepnumformat\csname @toc#1pnumformat\endcsname%
  \csname l@#1\endcsname}
\newcommand{\@tocsectionpnumformat}{}
\newcommand{\@tocsubsectionpnumformat}{}
\newcommand{\@tocsubsubsectionpnumformat}{}
\newcommand{\setsectionpnumformat}[1]{\renewcommand{\@tocsectionpnumformat}{#1}}
\newcommand{\setsubsectionpnumformat}[1]{\renewcommand{\@tocsubsectionpnumformat}{#1}}
\newcommand{\setsubsubsectionpnumformat}[1]{\renewcommand{\@tocsubsubsectionpnumformat}{#1}}
\renewcommand{\@tocpagenum}[1]{%
  \hfill {\mdseries\@sectypepnumformat #1}}

\let\oldappendix\appendix
\renewcommand{\appendix}{%
  \leavevmode\oldappendix%
  \addtocontents{toc}{%
    \protect\settowidth{\protect\@tocsectionnumwidth}{\protect\@tocsectionformat\sectionname\space}%
    \protect\addtolength{\protect\@tocsectionnumwidth}{2em}}%
}
\makeatother

\makeatletter
\settocsectionnumwidth{2em}
\settocsubsectionnumwidth{2.5em}
\settocsubsubsectionnumwidth{3em}
\settocsectionindent{1pc}%
\settocsubsectionindent{\dimexpr\@tocsectionindent+\@tocsectionnumwidth}%
\settocsubsubsectionindent{\dimexpr\@tocsubsectionindent+\@tocsubsectionnumwidth}%
\makeatother

\settocsectionvskip{2pt}
\settocsubsectionvskip{0pt}
\settocsubsubsectionvskip{0pt}

\settocsectionformat{\bfseries}
\settocsubsectionformat{\mdseries}
\settocsubsubsectionformat{\mdseries}
\setsectionpnumformat{\bfseries}
\setsubsectionpnumformat{\mdseries}
\setsubsubsectionpnumformat{\mdseries}

\let\oldtableofcontents\tableofcontents
\renewcommand{\tableofcontents}{%
  \vspace*{-5\linespacing}
  \oldtableofcontents}

\makeatletter
\let\origsection=\section \def\section{\@ifstar{\origsection*}{\mysection}} 
\def\mysection{\@startsection{section}{1}\z@{.7\linespacing\@plus\linespacing}{.5\linespacing}{\normalfont\scshape\centering\S}}
\makeatother  

\usepackage{amssymb}
\usepackage{amsmath}
\usepackage{amsthm}
\usepackage{letterswitharrows}
\usepackage{mathrsfs}
\usepackage{graphicx}
\usepackage{latexsym}
\usepackage{stmaryrd}
\usepackage[shortlabels]{enumitem}
\usepackage{moreenum}
\usepackage[bookmarks,hyperfootnotes=false, psdextra=true]{hyperref}
\usepackage{multirow}
\usepackage{xcolor}
\usepackage{caption}
\usepackage{subcaption}

\colorlet{darkishRed}{red!60!black}
\colorlet{darkishBlue}{blue!60!black}
\colorlet{darkishGreen}{green!50!black}
\colorlet{darkerishGreen}{green!30!black}
\colorlet{lightishGreen}{green!70!black}
\hypersetup{
    draft = false,
    bookmarksopen=true,
    colorlinks,
    linkcolor={darkishBlue},
    citecolor={darkishBlue},
    urlcolor={darkishBlue}
}
\usepackage[nameinlink, capitalise, noabbrev]{cleveref}
\usepackage{thmtools}
\usepackage{nameref}
\crefformat{enumi}{#2#1#3}
\crefformat{equation}{#2(#1)#3}
\crefrangeformat{section}{Sections~#3#1#4--#5#2#6}
\crefname{mainresult}{Theorem}{Theorems}
\usepackage{doi}
\let\setminus=\smallsetminus
\usepackage{tikz}
\usepackage{tikz-cd}
\usetikzlibrary{calc,through,intersections,arrows, trees, positioning, decorations.pathmorphing, cd}
\usepackage{comment}
\usepackage{mathtools}
\usepackage{nccmath}
\usepackage{pifont}
\usepackage[utf8]{inputenc}
\usepackage[T1]{fontenc}
\usepackage{lmodern}
\usepackage[babel]{microtype}
\usepackage[english]{babel}
\usepackage{relsize}

\newcommand{\COMMENT}[1]{{}}

\usepackage{geometry}
\let\setminus=\smallsetminus
\renewcommand{\leq}{\leqslant}
\renewcommand{\geq}{\geqslant}
\renewcommand{\ge}{\geq}
\renewcommand{\le}{\leq}

\let\rho=\varrho
\let\phi=\varphi

\newcommand{\id}{\normalfont\text{id}}

\DeclareMathOperator{\Aut}{Aut}

\newcommand{ \N } { \mathbb{N} }
\newcommand{ \Z } { \mathbb{Z} }

\newcommand{\defn}[1]{{\color{darkishGreen}{\emph{#1}}}}
\newcommand{\mathdefn}[1]{{\color{darkishGreen}{{#1}}}}

\makeatletter

\def\calCommandfactory#1{%
   \expandafter\def\csname c#1\endcsname{\mathcal{#1}}}
\def\frakCommandfactory#1{%
   \expandafter\def\csname frak#1\endcsname{\mathfrak{#1}}}

\newcounter{ctr}
\loop
  \stepcounter{ctr}
  \edef\X{\@Alph\c@ctr}
  \expandafter\calCommandfactory\X
  \expandafter\frakCommandfactory\X
  \edef\Y{\@alph\c@ctr}
  \expandafter\frakCommandfactory\Y
\ifnum\thectr<26
\repeat

\newcommand{\bT}{\mathbf{T}}

\setenumerate{label={\normalfont (\roman*)}}

\newtheorem{theorem}{Theorem}[section] 
\newtheorem{proposition}[theorem]{Proposition}
\newtheorem{corollary}[theorem]{Corollary}
\newtheorem{lemma}[theorem]{Lemma}

\newtheorem{claim}{Claim}
\crefname{claim}{Claim}{Claims}
\AtEndEnvironment{proof}{\setcounter{claim}{0}}

\theoremstyle{definition}

\newtheorem{example}[theorem]{Example}

\newtheorem*{definition*}{Definition}

\newtheorem{construction}[theorem]{Construction}

\theoremstyle{remark}

\usepackage{etoolbox}
\newbool{pdfBool}
\boolfalse{pdfBool} 

\newcommand{\pdfOrNot}[2]{\ifbool{pdfBool}{{#1}}{{#2}}}

\usepackage{svg}

\def\lqedsymbol{\ifmmode$\lrcorner$\else{\unskip\nobreak\hfil
		\penalty50\hskip1em\null\nobreak\hfil$\rule{1.2ex}{1.2ex}$
		\parfillskip=0pt\finalhyphendemerits=0\endgraf}\fi}

\newenvironment{claimproof}[1][\proofname]
{%
	\proof[#1]%
}
{%
	\endproof%
}

\newcommand{\td}{tree-decom\-pos\-ition}
\newcommand{\gd}{graph-decom\-pos\-ition}

\usepackage{tikz}
\usetikzlibrary{positioning,shapes,calc}

\usepackage{csquotes}
\usepackage[backend=biber, style=alphabetic, maxnames=99, maxalphanames=99, giveninits=true, sortcites=true]{biblatex}
\DeclareLabelalphaTemplate{
    \labelelement{
        \field[final]{shorthand}
        \field{label}
        \field[strwidth=1,strside=left]{labelname}
    }
    \labelelement{
        \field[strwidth=2,strside=right]{year}
    }
}

\begin{document}

\title{Canonical tree-decompositions of chordal graphs}

\author[R.\ W.\ Jacobs \and P.\ Knappe]{Raphael W. Jacobs, Paul Knappe}
\address{Universität Hamburg, Department of Mathematics, Bundesstraße 55 (Geomatikum), 20146 Hamburg, Germany}
\email{\{raphael.jacobs, paul.knappe\}@uni-hamburg.de}

\thanks{R.W.J. is supported by doctoral scholarships of the Studienstiftung des deutschen Volkes and the Cusanuswerk -- Bisch\"{o}fliche Studienf\"{o}rderung. P.K. is supported by a doctoral scholarship of the Studienstiftung des Deutschen Volkes.}


\begin{abstract}
    We show that a locally finite, connected graph~$G$ is $r$-locally chordal (that is, its $r/2$-balls are chordal) if and only if the unique canonical graph-decomposition~$\cH_r(G)$ of~$G$ displaying its $r$-global structure is into cliques.
    Our proof relies on a canonical version of Halin's characterization of chordal locally finite graphs as those that admit a tree-decomposition into cliques:
    We show that such tree-decompositions can be chosen to be canonical, that is, so that they are invariant under all the graph's automorphisms.
\end{abstract}

\maketitle

\section{Introduction}

\subsection{Characterising \texorpdfstring{$r$}{r}-locally chordal graphs}

A recent development in graph minor theory aims to study the local and global aspects of a graph's connectivity structure through the lens of `local coverings'~\cite{canonicalGD,computelocalSeps}.
Here, the \emph{$r$-local covering} \hbox{$p_r: G_r \to G$} of a graph~$G$ is a (topological) covering of~$G$ that preserves all $r$-local structure of~$G$ -- that is, its $r/2$-balls -- while unfolding all remaining structure in a tree-like way (see~\cref{sec:Background:LocalCoverings} for the formal definition).
In this definition, the integer parameter $r \ge 0$ specifies the intended degree of locality.

Diestel, Jacobs, Knappe and Kurkofka use $r$-local coverings to bring tangle theory from graph minors to bear canonically on arbitrary graphs, which need not be tree-like~\cite{canonicalGD}.
Given a graph $G$ and an integer $r > 0$, they construct a canonical graph-decomposition $\cH_r(G)$ of $G$ that displays the $r$-global structure of~$G$.
Here, \emph{graph-decompositions} generalise tree-decompositions in that the decomposition tree may take the form of an arbitrary graph (see~\cref{sec:Background:GraphDecsAndCoverings} for the formal definition).
Further, a graph-decomposition is \emph{canonical} if it is invariant under the automorphisms of the graph.

\begin{restatable}{theorem}{CanGraphDec}\emph{({\cite[Theorem 1]{canonicalGD}})} \label{thm:CanGraphDec}
    Let $G$ be a finite graph, and $r > 0$ an integer.
    Then $G$ has a unique canonical decomposition~$\cH_r(G)$ modelled on another finite graph $H = H(G, r)$ that displays its $r$-global structure.
\end{restatable}

Following~\cite{canonicalGD}, the construction of~$\cH_r(G)$ applies state-of-the-art tangle theory to the $r$-local cover~$G_r$ of~$G$ to obtain a canonical \td\ $\cT(G_r)$ of~$G_r$ that suitably captures the connectivity structure of~$G_r$; such a \td\ is well suited for the covering~$G_r$ due to its global tree-like nature.
The \td\ $\cT(G_r)$ is then `folded' along the covering map~$p_r: G_r \to G$ to obtain a canonical graph-decomposition $\cH_r(G)$ of~$G$ that displays the $r$-global structure of~$G$ (see~\cref{sec:Background:GraphDecsAndCoverings,sec:CanGraphDecIntoCliques}).

It is also possible to reconstruct the graph-decomposition~$\cH_r(G)$ from finite information in $G$ that is sufficiently local to be reflected in the covering graph~$G_r$~\cite{computelocalSeps}.
This construction avoids the use of the usually infinite $r$-local covering and instead builds on a finite theory of $r$-local separations (see~\cref{sec:CanGraphDecIntoCliques}).\\

One of the natural questions arising from~\cref{thm:CanGraphDec} is whether there are classes of graphs~$G$ that can be characterised by their decompositions~$\cH_r(G)$.
With the main result of this paper, we answer this question affirmatively for the class of `$r$-locally chordal' graphs, thereby obtaining the first known such characterisation.

For this, given an integer~$r \ge 0$, a graph is \defn{$r$-locally chordal} if all its balls of radius $r/2$ are chordal~\cite{LocallyChordal}.
Further, a graph-decomposition of $G$ is \defn{into cliques} if all of its bags are cliques in $G$, that is, they induce complete subgraphs of $G$.
Our main result now reads as follows:

\begin{restatable}{mainresult}{Hrintcliques}\label{thm:Hr-int-cliques}
   Let $G$ be a finite, connected graph and $r \geq 3$ an integer.
   Then the following are equivalent:
   \begin{enumerate}
       \item $G$ is $r$-locally chordal.
       \item The canonical \gd\ $\cH_r(G)$ of~$G$ which displays its $r$-global structure is into cliques.
   \end{enumerate}
\end{restatable} 

\cref{thm:Hr-int-cliques} can be seen as an $r$-local analogue of Halin's well-known characterisation of locally finite, connected, chordal graphs as those that admit a tree-decomposition into cliques~\cite{halin1964simpliziale,halin1989graphentheorie}.
As such it fits into the global structure theorems for $r$-locally chordal graphs by Abrishami and Knappe~\cite{localGlobalChordal}, who also use the results in this paper (see~\cref{sec:GDsIntoMaximalCliques}).
The definition of $r$-locally chordal graphs, the construction of~$\cH_r(G)$, and the proof of~\cref{thm:Hr-int-cliques} all follow the perspective of $r$-local coverings as a framework to discriminate between the local and global aspects of a graph's structure, as we now explain.

Recall that a graph is \defn{chordal} if every cycle of length at least four has a chord; in particular, a graph is chordal if and only if it is $r$-locally chordal for all integers $r > 0$.
One can show that a graph $G$ is $r$-locally chordal if and only if its $r$-local cover $G_r$ is chordal~\cite[Theorem~1]{LocallyChordal}; in other words, the definition of `$r$-locally chordal' aligns with the perspective of $r$-local coverings.
Abrishami, Knappe and Kobler also exhibit $r$-local analogues of further classical characterisations of chordal graphs, such as by their minimal separators~\cite{LocallyChordal}.

For the proof of~\cref{thm:Hr-int-cliques}, recall that the graph-decomposition $\cH_r(G)$ is obtained by folding the canonical tree-decomposition $\cT(G_r)$ along the covering map $p_r$.
Building on this construction, our proof of~\cref{thm:Hr-int-cliques} follows the `local-global paradigm' of $r$-local coverings:
We do not work in $G$ itself, but instead build on the $r$-local covering $p_r: G_r \to G$ and on global structure theorems which we apply to $G_r$ and then transfer to $G$ via $p_r$, as follows.

If a finite, connected graph $G$ is $r$-locally chordal, then its $r$-local cover $G_r$ is a locally finite, connected, chordal graph.
We show that $\cT(G_r)$ is then a canonical \td\ into cliques (\cref{thm:canonicalchordalTD}), building on the construction in~\cite{computelocalSeps}; this strengthens the known existence of such tree-decomposition of locally finite, connected, chordal graphs~\cite{halin1964simpliziale,halin1989graphentheorie} to ones that are even canonical (see~\cref{subsec:Intro:TDsChordal}).
Finally, we prove that $\cH_r(G)$ is into cliques if and only if $\cT(G_r)$ is into cliques, which then completes the proof of~\cref{thm:Hr-int-cliques}.

\subsection{Tree-decompositions and chordal graphs} \label{subsec:Intro:TDsChordal}

It is well-known\footnote{This statement is immediate; e.g. see the first paragraph of the proof presented in \cite[Theorem~3]{GavrilChordalGraphsSTIG}.} that graphs admitting \td s into cliques\footnote{In this paper, we investigate whether a graph admits certain \td s into cliques. We remark that all these questions can equivalently be formulated in the language of subtree intersection graphs (over trees); see~\cite{localGlobalChordal}.} are chordal.
Recall that a \td\ $(T, (V_t)_{t \in V(T)})$ of a graph $G$ is \defn{into cliques} if all its bags $V_t$ are cliques in $G$, that is, every~$G[V_t]$ is a complete subgraph of~$G$.

\begin{proposition}[Folklore]\label{prop:subtreeintersec-is-chordal}
    Every (possibly infinite) graph that admits a \td\ into cliques is chordal.
\end{proposition}

For finite graphs, it is also well-known that the converse of \cref{prop:subtreeintersec-is-chordal} holds, too (see e.g.\ Gavril \cite[Theorem~2 and~3]{GavrilChordalGraphsSTIG}): the chordal graphs are precisely those graphs admitting \td s into cliques.
This result was later extended to graphs with no infinite clique, such as locally finite graphs, by Halin \cite[10.2']{halin1964simpliziale} based on his \emph{zweiter Zerlegungssatz} \cite[Satz 5]{halin1989graphentheorie}:

\begin{theorem}[Halin]\label{thm:NoInifniteCliqueChordalIsTDIntoCliques}
    A graph $G$ with no infinite clique is chordal if and only if it admits a \td\ into (its maximal) cliques.
\end{theorem}

Diestel~\cite[Theorem~4.2]{CharacterisationOfCountableChordalGraphsAdmitTD} even characterised the countable graphs that admit a \td\ into cliques via forbidden `simplicial minors'.
This characterisation does not extend to uncountable graphs \cite[Paragraphs after Theorem 1.4]{CounterexampleUncountableChordalAdmitTD}, but Diestel conjectured an extension of his characterisation~\cite[End of \S1]{CounterexampleUncountableChordalAdmitTD}.\\

When working with highly symmetric graphs, such as Cayley graphs or covers of (finite) graphs, it is often desirable to find canonical \td s, as in the proof of~\cref{thm:Hr-int-cliques} or e.g.\ in~\cite{canonicalGD,CanTreesofTDs,StallingsQuasiTransGraphs,dunwoody2015vertex}.
Formally, a \td\ $(T, (V_t)_{t \in V(T)})$ of a graph $G$ is \defn{$\Gamma$-canonical} for a subgroup~$\Gamma$ of the automorphism-group~$\Aut(G)$ of~$G$ if there is an action of~$\Gamma$ on~$T$ that commutes with $t \mapsto V_t$, i.e.\ $V_{\phi \cdot t} = \phi(V_t)$ for every $t \in V(T)$ and $\phi \in \Gamma$; usually, the action of~$\Gamma$ on~$T$ is unique (\cref{lem:uniqueaction-on-regular-TD}).
If a \td\ is~$\Aut(G)$-canonical, then we call it \defn{canonical} for short.

Does~\cref{thm:NoInifniteCliqueChordalIsTDIntoCliques} still hold if we want the tree-decomposition into cliques to be canonical?
In other words, does every chordal graph with no infinite clique admit a canonical tree-decomposition into cliques?
With~\cref{thm:canonicalchordalTD} we give an affirmative answer for chordal graphs that are locally finite and connected\footnote{Disconnected chordal graphs sometimes, e.g. the disjoint union of two paths of length $3$, do not admit \emph{canonical} \td s into cliques for the simple reason that \td s do not allow forests as model graphs. 
We remark that a disconnected chordal graph admits a canonical forest-decomposition into cliques if and only if its components admit canonical \td s into cliques. 
Hence, we restrict our view in this paper to connected graphs and tree-decompositions.}; this in particular includes the $r$-local cover $G_r$ of a locally finite, connected graph $G$.

\begin{restatable}{mainresult}{canonicalchordalTD}\label{thm:canonicalchordalTD}
    Let $G$ be a locally finite, connected graph. 
    Then $G$ is chordal if and only if $G$ admits a canonical \td\ into cliques.
    More precisely, $\cT(G)$ is a canonical \td\ into cliques.
\end{restatable}

\noindent For more context on the proof of~\cref{thm:canonicalchordalTD}, see~\cref{subsec:Intro:ProofChordalTD}.

\cref{thm:NoInifniteCliqueChordalIsTDIntoCliques} makes the even stronger assertion that there are \td s \emph{into the graph's maximal cliques}.
Formally, a \defn{maximal} clique of a graph~$G$ is an inclusion-wise maximal clique, and a \td\ $(T, (V_t)_{t \in V(T)})$ of $G$ is \defn{into its maximal cliques} if the map $t \mapsto V_t$ is a bijection from~$V(T)$ to the set of maximal cliques of $G$.
We show that~\cref{thm:canonicalchordalTD} cannot be strengthened to canonical \td s into maximal cliques:
there are chordal graphs, even finite and connected ones, that do not admit a canonical \td\ into their maximal cliques (\cref{ex:no-maximal-and-canonical}).

In contrast, we provide a positive result for graphs~$\hat G$ from normal coverings~$p: \hat G \to G$ of locally finite, connected graphs~$G$ (see~\cref{sec:TDCoveringsMaxCliques} for the definitions); this in particular includes $r$-local coverings $p_r: G_r \to G$.
If such~$\hat G$ is chordal, then it admits a \td\ into its maximal cliques which is at least~$\Gamma(p)$-canonical for the group~$\Gamma(p)$ of deck transformations of~$p$:

\begin{restatable}{mainresult}{chordalTDmaximal}\label{thm:chordalTDmaximal}
    Let $G$ be a locally finite, connected graph and $p \colon \hat G \to G$ a normal covering.
    If $\hat G$ is chordal, then $\hat G$ admits a $\Gamma(p)$-canonical \td\ into its maximal cliques.
\end{restatable}

\noindent Analogous to the use of~\cref{thm:canonicalchordalTD} in the proof of~\cref{thm:Hr-int-cliques}, the even stronger~\cref{thm:chordalTDmaximal} can be used to obtain a \gd\ of an $r$-locally chordal graph~$G$ which is into maximal cliques (see~\cref{sec:GDsIntoMaximalCliques}).

We also discuss what happens behind the barrier of locally finite graphs:
there are countable chordal graphs with no infinite clique that admit a \td\ into cliques, but no canonical one (\cref{ex:CountableNoCanonical}), and there are countable chordal graphs that admit a \td\ into cliques, but no \td\ into their maximal cliques (\cref{ex:CountableNoMaxCliques}).

\subsection{On the proof of~\texorpdfstring{\cref{thm:canonicalchordalTD}}{Theorem 2}} \label{subsec:Intro:ProofChordalTD}

Recall that the construction of~$\cH_r(G)$ described above builds on the construction of a canonical tree-decomposition $\cT(G_r)$ of the $r$-local covering $G_r$ of $G$.
So in order to apply~\cref{thm:canonicalchordalTD} in the proof of our main result, \cref{thm:Hr-int-cliques}, it is crucial to show the `more precisely'-part of~\cref{thm:canonicalchordalTD}, that is, that $\cT(G_r)$ is a canonical \td\ into cliques of a locally finite, connected, chordal graph $G_r$; in other words, merely the existence of such a \td\ is not enough for our purpose. 
While the construction of $\cT(G)$ for a graph $G$ relies on the recent notion of \emph{bottlenecks} introduced in~\cite{computelocalSeps} (see~\cref{construction:N(G)}), it follows the state-of-the-art method for building trees-of-tangles, which has been more and more refined over the past years; see in particular~\cite{finitesplinter,infinitesplinter,entanglements,computelocalSeps}.

In its current form, this method does not consider tangles directly, but sets of efficient tangle-distinguishers axiomatised as bottlenecks.
The tree-of-tangles is then constructed as a nested set of separations meeting all bottlenecks, and this nested set is built iteratively, by adding suitable separations of increasing order.
In infinite, locally finite graphs, such a nested set does not necessarily induce a tree-decomposition of the graph (see~\cite{infinitesplinter,ExamplesLinkedLeanTD}).
For chordal graphs, however, this issue disappears since all separations arising in the construction are tight and their separators are therefore cliques.
It remains to verify that the parts of the resulting \td\ are cliques; this follows, informally, because each clique defines a tangle of the graph, and distinguishing these clique‑tangles forces every part to lie inside a clique.

While the `more precisely'-part of~\cref{thm:canonicalchordalTD} is crucial for its application in the proof of~\cref{thm:Hr-int-cliques}, one can alternatively express the construction of a \td\ as in the main statement of~\cref{thm:canonicalchordalTD} more directly in the language of~\cite{CanTreesofTDs}, as follows.
We first cut the graph along all its tight clique-separators of size~$1$, then further decompose the resulting parts along their tight clique-separators of size~$2$, and continue iteratively.
This yields a so-called `canonical tree of tree-decompositions', which can be transformed into a canonical tree-decomposition when the underlying graph is locally finite.
As before, one can then verify that the parts of this \td\ are cliques if the graph is chordal.

The proof of~\cref{thm:chordalTDmaximal} applies an iterative post-processing to the \td\ from~\cref{thm:canonicalchordalTD}.
This becomes possible because this \td\ is canonical and the decomposed graph is a covering graph; the key difficulty lies in maintaining throughout the iteration that the \td\ is into cliques.

\subsection{Structure of the paper}

In~\cref{sec:Background} we recall the definitions and concepts around separators, separations and \td s and prove a few simple lemmas for later use.
We then prove~\cref{thm:canonicalchordalTD,thm:chordalTDmaximal} in~\cref{sec:CanTdIntoCliques}.
In~\cref{sec:COMPcanonicalGDviaCoverings}, we then apply these result to $r$-locally chordal graphs and prove \cref{thm:Hr-int-cliques}.
Finally, we provide counterexamples to various strengthenings of~\cref{thm:canonicalchordalTD} in~\cref{sec:CounterEx}.

\subsection{Acknowledgements}

We thank Tara Abrishami for many stimulating discussions on locally chordal graphs and local coverings.
We thank Mai Trinh for useful comments on the paper.

\section{Background} \label{sec:Background}

We follow standard graph theory notation and definitions from \cite{bibel}. We also refer to 
\cite{Hatcher} for terminology on topology. 

In this paper we deal with simple graphs, i.e.\ graphs with neither loops nor parallel edges.
Often, we will restrict our view to locally finite graphs; a graph is \defn{locally finite} if each vertex has finite degree. We specify when we deal with finite, locally finite, or general (possibly infinite) graphs.

For a set $\cX$ of sets, we denote $\bigcup_{X \in \cX} X$ by \defn{$\bigcup \cX$}.

\subsection{Minimal, tight and efficient separators}

Let $u,v$ be two vertices and $X$ a vertex-subset of a graph~$G$.
Then we say that $X$ \defn{separates $u$ and $v$} in $G$ and call $X$ an \defn{$u$--$v$ separator} of $G$ if $X \subseteq V(G) \setminus \{u,v\}$ and every $u$--$v$ path in $G$ meets $X$.
If $X$ separates some two vertices in $G$, then~$X$ is a \defn{separator} of $G$.
A separator $X$ is a \defn{minimal (vertex) separator} of $G$ if there are two vertices $u, v$ of $G$ such that $X$ is inclusion-minimal among the $u$--$v$ separators in $G$. 

Dirac characterized chordal graphs by their minimal separators: 

\begin{restatable}[Dirac~\cite{dirac1961rigid}, {Theorem~1}]{theorem}{tightsepchar}
    \label{thm:CharacterisationChordalViaMinimalVertexSeparator}
    A (possibly infinite)\footnote{It is immediate from the proof presented in \cite[Theorem~2.1]{blair1993introduction} that this also holds for infinite graphs.} graph is chordal if and only if every minimal separator is a clique.
\end{restatable} 

A vertex-subset $X$ of a graph $G$ is a \defn{tight} separator of $G$ if at least two components~$C$ of~$G-X$ are \defn{full}, i.e.\ $N_G(C) = X$.\footnote{Full components are sometimes also called {\em tight} components.}
It is easy to see that a separator $X$ of $G$ is tight if and only if it is minimal.
The relation between a minimal (equivalently: tight) finite separator in a chordal graph and the respective full components is quite restrictive; indeed, we can extract the following structure from a proof of Hofer-Temmel and Lehner~\cite{chordalinfinite}:

\begin{lemma}\label{thm:CliqueAtTightSep}
    Let $X$ be a finite clique in a (possibly infinite) chordal graph $G$.
    Then for every full component~$C$ of~$G-X$, there is a vertex in~$C$ complete to $X$ in $G$.
\end{lemma}
\begin{proof}
    Follow Hofer-Temmel and Lehner's proof of \cite[Lemma 14]{chordalinfinite} to obtain the desired result; there, the existence of the paths~$P_w$ is guaranteed for every $w \in X$ by $C$ being a full component of~$G-X$.
\end{proof}

\noindent
We remark that \cref{thm:CliqueAtTightSep} fails if we allow the clique~$X$ to be infinite\footnote{The statement of \cite[Lemma 3.13]{chordalinfinite} does not include this necessary assumption of~$X$ being finite. If~$X$ is infinite, then their proof fails in that there might not even exist a single inclusion-maximal element among the~$C_v$ for $v \in V_1$.}:
Consider the chordal graph~$G$ consisting of two countable infinite cliques~$X = \{x_0, x_1, \dots\}$ and~$C = \{c_0, c_1, \dots \}$ that are joined by all edges~$x_i c_j$ with~$i \le j$.
Then~$C$ is a full component of~$G-X$, but every~$c_j \in X$ has only the finite set~$\{x_i : i \le j\}$ of neighbours in~$X$, contradicting the statement.

Finally, two maximal cliques in a chordal graph can also be told apart by vertex sets of comparatively smaller size.
To formalize this, let $A,B,X$ be vertex-subsets of a graph $G$.
Then we say that $X$ \defn{separates $A$ and $B$} in $G$ and call $X$ an \defn{$A$--$B$ separator} of $G$ if each $A$--$B$ path in $G$ meets $X$.
If the size of $X$ is minimal among all the $A$--$B$ separators, then $X$ \defn{efficiently} separates $A$ and $B$.

Now an efficient separator of any two finite, maximal cliques in a chordal graph is smaller than the minimum of their sizes:

\begin{theorem}\label{thm:smallEffSepOfMaxCliques}
    Let $X$ and $Y$ be two distinct maximal cliques in a (possibly infinite) chordal graph $G$.
    If at least one of $X$ and $Y$ is finite, then there is a $X$--$Y$ separator of $G$ whose size is less than the sizes of $X$ and $Y$.
\end{theorem}

\noindent As a tool for the proof of~\cref{thm:smallEffSepOfMaxCliques}, we need the following version of Menger's Theorem for infinite graphs, which is due to Erd\H{o}s and was published in \cite{KonigTheorieDerEndlichenUndUnendlichen}.

\begin{theorem}\label{thm:Menger}
    Let $A$ and $B$ be vertex-subsets of a (possibly infinite) graph $G$.
    Then the cardinality of an efficient $A$--$B$ separator of $G$ (equivalently: the minimum cardinality of an $A$--$B$ separator) is the maximum cardinality of a family $\cP$ of disjoint $A$--$B$ paths in $G$.
\end{theorem}
\begin{proof}[Proof of~\cref{thm:smallEffSepOfMaxCliques}]
    By symmetry of the statement in $X$ and $Y$, we may assume without loss of generality that $|X| \leq |Y|$.
    Since $X$ and $Y$ are distinct maximal cliques, there is a vertex $y \in Y \setminus X$.
    Thus, the component~$C$ of~$G - X$ which meets~$Y$ is non-empty; note that $C$ is unique and contains~$Y \setminus X$, since~$Y$ is a clique.
    
    Suppose for a contradiction that there is no $X$--$Y$ separator of size less than $|X|$.
    We claim that~$C$ then is a full component of~$G-X$.
    Indeed, by~\cref{thm:Menger}, there is a family of~$|X|$ disjoint $X$--$Y$ paths in~$G$, and all of them must be contained in~$C$ except for their endvertices in~$X$, thus witnessing that~$C$ is a full component of~$G-X$.
    Now we apply~\cref{thm:CliqueAtTightSep} to obtain a vertex~$w \in C$ complete to~$X$.
    This then contradicts the maximality of the clique $X$ in $G$, as $X \cup \{w\}$ forms a larger clique.
\end{proof}

\begin{corollary} \label{lem:EffCliqueDistAreTight}
    Every efficient separator of two distinct maximal cliques in a (possibly infinite) chordal graph is minimal and tight.
\end{corollary}
\begin{proof}
    Let~$S$ be an efficient separator of two maximal cliques~$X$ and~$Y$.
    By~\cref{thm:smallEffSepOfMaxCliques}, $S$ has size less than the sizes of~$X$ and~$Y$.
    Thus, there exist vertices~$x \in X \setminus S$ and~$y \in Y \setminus S$.
    Then $S$ is also a~$x$--$y$ separator, and it is in fact a minimal one since each~$x$--$y$ separator also separates~$X \ni x$ and~$Y \ni y$ and~$S$ did so efficiently.
\end{proof}

\subsection{Tree-decompositions}

Let $G$ be a graph, let $T$ be a tree, and let $\cV = (V_t \mid t \in V(T))$ be a family of vertex-subsets of $G$ indexed by the nodes of~$T$. 
The pair $(T, \cV)$ is a \defn{\td} of $G$ if
\begin{enumerate}[label=\rm{(T\arabic*)}]
    \item\label{T1} $G = \bigcup_{t \in V(T)} G[V_t]$, and
    \item\label{T2} for every vertex $v$ of $G$, the subgraph \defn{$T_v$} of $T$ induced by the set of nodes of $t$ whose bags $V_t$ contain $v$ is connected.
\end{enumerate}
For a tree-decomposition $(T, \cV)$, we call $T$ its \defn{decomposition tree}, the~$V_t$ its \defn{bags} and the induced subgraphs~$G[V_t]$ its \defn{parts}.
For an edge $f = t_1 t_2 \in E(T)$, the set~$\mathdefn{V_f} \coloneqq V_{t_1} \cap V_{t_2}$ is the \defn{adhesion set} corresponding to~$f$.
We often introduce a tree-decomposition as a pair~$(T, \cV)$ and tacitly assume that~$\cV = (V_t \mid t \in V(T))$.

Roughly speaking,
\cref{T1} says that $\cV$ forms a decomposition of $G$, and, by \cref{T2}, the decomposition $\cV$ follows the structure of the tree $T$.
The second axiom can be made more descriptive by characterising a \td\ $(T,\cV)$ in terms of the separations of $G$ induced by the edges of $T$:
A pair $(T,\cV)$ is a \td\ if and only if it satisfies \cref{T1} and
\begin{enumerate}[label=(T\arabic*')]
    \setcounter{enumi}{1}
    \item\label{T2'} For every edge $f = t_1t_2$ of $T$ and for $i = 1,2$, set $A^f_i \coloneqq \bigcup_{t \in V(T_i)} V_t$ where $T_i$ is the component of $T-f$ which contains~$t_i$. Then $\mathdefn{s^f} \coloneqq\{A^f_1,A^f_2\}$ is a separation of $G$ and its separator $A^f_1 \cap A^f_2$ is the adhesion set $V_f = V_{t_1} \cap V_{t_2}$ corresponding to $f$.
\end{enumerate}

A \td\ $(T,\cV)$ of $G$ is \defn{into cliques} if all its bags are cliques in $G$.
For the interested reader, we remark that a map which assigns to each vertex $v$ of $G$ a subgraph $T_v$ of a tree $T$ is a subtree representation of $G$  if and only if $(T,\cV)$ is a \td\ into cliques with $V_t \coloneqq \{v \in V(G) \mid t \in V(T_v)\}$. See \cite[\S 2]{localGlobalChordal} for a more general correspondence.

A simple observation is that any finite clique is contained in a single bag of a \td:

\begin{lemma}\label{lem:CliqueContainedInBag}
    Let $(T,\cV)$ be a \td\ of a (possibly infinite) graph $G$.
    For every finite\footnote{For infinite cliques \cref{lem:CliqueContainedInBag} fails trivially: A countably infinite clique~$\{v_1, v_2, \dots\}$ has a tree-decomposition~$\{T, \cV\}$ into finite bags where~$T = t_1 t_2 \dots$ is a ray and~$V_{t_i} = \{v_1, \dots, v_i\}$.} clique $K$ in $G$, there is a node $t$ of $T$ such that the bag $V_t$ contains $K$.
\end{lemma}
\begin{proof}
    Consider the subtrees $T_v$ of $T$ from \cref{T2} with $v \in K$.
    By \cref{T1}, the $T_v$ are pairwise intersecting.
    A folklore result asserts that every family of subtrees of a (possibly infinite) tree~$T$ is (finitely) Helly, i.e.\ whenever a finite family $\bT$ of subtrees of~$T$ is pairwise intersecting, then their intersection $\bigcap \bT$ is non-empty.
    Since $K$ is finite, this yields that there is a node $t$ that is contained in every $T_v$ with $v \in K$, i.e.\ $V_t$ contains~$K$, as desired.
\end{proof}

\subsection{Interplay of tree-decompositions and separations}\label{subsec:interplayTDandSEP}

First, let us recall some definitions regarding separations.
Let $G$ be a graph.
A separation $\{A,B\}$ of $G$ is \defn{proper} if both $A \setminus B$ and $B \setminus A$ are non-empty.
A separation $\{A,B\}$ of $G$ is \defn{tight} if both $A \setminus B$ and $B \setminus A$ contain the vertex set of a full component of $G-(A \cap B)$.
Analogous to separators, it is easy to see that a separation $\{A,B\}$ is tight if and only if its separator $A \cap B$ is an inclusion-minimal $a$--$b$ separator for some $a \in A \setminus B$ and $b \in B \setminus A$.
The \defn{order} of $\{A,B\}$ is the size $k \coloneqq |A \cap B|$ of its separator;
we also say that $\{A,B\}$ is a \defn{$k$-separation} of $G$.

We refer to $(A,B)$ and $(B,A)$ as the \defn{orientations} of the separation $\{A,B\}$ or just as \defn{oriented separations} of $G$.
Two separations $\{A,B\}$ and $\{C,D\}$ are \defn{nested} if there are orientations of $\{A,B\}$ and $\{C,D\}$, say $(A,B)$ and $(C,D)$, such that $A \subseteq C$ and $B \supseteq D$.
A set of separations is \defn{nested} if its elements are pairwise nested.

Recall that a \td\ $(T,\cV)$ of a graph $G$ is \defn{$\Gamma$-canonical} for a subgroup~$\Gamma$ of~$\Aut(G)$ if there is an action of $\Gamma$ on $T$ that commutes with $t \mapsto V_t$, i.e.\ $V_{\phi \cdot t} = \phi(V_t)$ for every $t \in V(T)$ and $\phi \in \Gamma$; we will usually write~$\phi(t)$ instead of~$\phi \cdot t$ for notational simplicity.
An $\Aut(G)$-canonical \td\ is also called \defn{canonical} for short.
Furthermore, recall that, by \cref{T2'}, every edge $f$ of the decomposition tree $T$ induces a separation $s^f$ of $G$.
A \td\ is \defn{regular} if all its induced separations are proper.

\begin{lemma}[\cite{canonicalGD}, Lemma 3.7] \label{lem:uniqueaction-on-regular-TD}
    Let $(T,\cV)$ be a \td\ of a graph $G$.
    If $(T,\cV)$ is regular, then for every automorphism $\gamma$ of $G$ there exists at most one automorphism $\phi$ of $T$ such that $\gamma(V_t) = V_{\phi(t)}$.
    In particular, if additionally $(T,\cV)$ is $\Gamma$-canonical for a subgroup $\Gamma$ of $\Aut(G)$, then there is a unique action of $\Gamma$ on $T$ which commutes with $t \mapsto V_t$.
\end{lemma}

Given a \td\ $\cT \coloneqq (T, \cV)$ of a graph~$G$, it follows from the definition that the set~$\mathdefn{N(\cT)} \coloneqq \{s^f : f \in E(T)\}$ of all its induced separations~$s^f$ of~$G$ is nested.
If~$\cT$ is also $\Gamma$-canonical for a subgroup~$\Gamma$ of~$\Aut(G)$, then the set~$N(\cT)$ inherits this symmetry in that it is itself~\emph{$\Gamma$-invariant}:
The automorphisms~$\phi$ of $G$ naturally act on the separations $s = \{A,B\}$ of $G$ by the map $\phi(s) \coloneqq \{\phi(A),\phi(B)\}$.
A set $S$ of separations of $G$ is \defn{$\Gamma$-invariant}\footnote{$\Gamma$-invariant sets of separations are sometimes also called \emph{$\Gamma$-canonical}, and $\Aut(G)$-invariant sets of separations are referred to as \emph{canonical} for short (see e.g.~\cite{canonicalGD}).} if $\phi(s) \in S$ for every $s \in S$ and every $\phi \in \Gamma$.

Conversely, in finite graphs, a ($\Gamma$-invariant) nested set of separations of~$G$ induces a ($\Gamma$-canonical) \td\ of~$G$.
In infinite graphs, this holds under an additional assumption, as shown in~\cite[Lemma~2.7]{infinitesplinter}; we here state their lemma in a way compatible with~\cite{computelocalSeps}. 
For this, let us call a set $S$ of separations of $G$ \defn{point-finite} if every vertex of $G$ appears in the separator of at most finitely many separations in $S$.
Further, a \td\ $(T,\cV)$ of $G$ is \defn{point-finite} if $T_v$ is finite for every vertex $v$ of $G$.
Given a nested set $N$ of proper separations of a graph, let \defn{$\cT(N)$} be defined as in \cite[Construction~2.3]{computelocalSeps}.

\begin{lemma}[{\cite[Lemma~2.7]{infinitesplinter}}]\label{lem:PointFiniteInducesTD}
    Let $N$ be a nested set of proper separations of a connected graph $G$.
    If $N$ is point-finite, then $\cT(N)$ is a \td\ of $G$ such that $f \mapsto s^f$ is a bijection from the edges of $T$ to $N$. 
    In particular, $\cT(N)$ is regular and point-finite. 
    Moreover, if $N$ is $\Gamma$-invariant for a subgroup~$\Gamma$ of~$\Aut(G)$, then $\cT(N)$ is $\Gamma$-canonical.
\end{lemma}

\section{Canonical tree-decompositions into cliques}\label{sec:CanTdIntoCliques}

In this section, we prove \cref{thm:canonicalchordalTD,thm:chordalTDmaximal}.

\subsection{Proof of \texorpdfstring{\cref{thm:canonicalchordalTD}}{Theorem 2}} \label{sec:ConstructionN(G)}

Let us restate~\cref{thm:canonicalchordalTD}:

\canonicalchordalTD*

\noindent
The backwards-implication is straightforward (\cref{prop:subtreeintersec-is-chordal}).
For the forwards-implication, we prove the `more precisely'-part of the statement.
Following~\cite{computelocalSeps}, $\cT(G)$ is defined as $\mathdefn{\cT(G)} := \cT(N(G))$ for a specific nested set~$N(G)$ of separations of $G$.

The construction of~$N(G)$ draws on the most recent developments in tangle-theory.
We shall not need any formal details here, but use these results as a black-box.
Let us give a rough sketch of our construction in the context of tangle-theory.
Tangles are an abstract notion of highly connected substructures of graphs originating in the work of Robertson and Seymour on their Graph Minor Theorem~\cite{GM}.
One of the central results of tangle-theory is the so-called Tree-of-Tangles Theorem which asserts that the relative position of the tangles of a graph to each other is tree-like and can be captured by a nested set of separations~\cite{GM}.

A \emph{tree of tangles} is a nested set of separations with the property that every two tangles of the graph are efficiently distinguished by one of its elements; many modern constructions even yield \emph{canonical} trees of tangles, that is, $\Aut(G)$-invariant such nested sets.
The most recent methods in tangle-theory find trees of tangles by focusing on the sets of separations that efficiently distinguish two given tangles~\cite{infinitesplinter,finitesplinter,entanglements,computelocalSeps}.
Such sets of efficient tangle-distinguishers can themselves be axiomatized in a similar manner as tangles, leading to the notion of \emph{bottlenecks}~\cite{computelocalSeps}.
The quest to find a (canonical) tree of tangles of a graph then translates into finding a (canonical) nested set of separations which meets all of the graph's bottlenecks.

To find the nested set~$N(G)$ of separations, we follow this approach:
We will first show that the set of separations that efficiently distinguishes two maximal cliques in a locally finite, chordal graph is indeed a bottleneck.
Secondly, we deduce that there is a canonical nested set~$N(G)$ of separations which meets all such bottlenecks; this step draws on recent results about tree-of-tangles constructions via bottlenecks~\cite{computelocalSeps}.
Thirdly, we show that~$\cT(G) = \cT(N(G))$ is indeed a \td\ of~$G$ and that this \td\ is into cliques.
Intuitively, if this \td\ were not into cliques, then some part would contain two distinct maximal cliques and they must be efficiently distinguished by a separation in~$N(G)$, which yields a contradiction. \\

Given two vertex-subsets $X,Y$ of a graph $G$, a separation $\{A,B\}$ of $G$ \defn{distinguishes} $X$ and $Y$ if, up to renaming sides, $X \subseteq A$ and $Y \subseteq B$.\footnote{In this paper, we always consider separations~$\{A, B\}$ that efficiently distinguish two distinct maximal cliques $X$ and $Y$, and by~\cref{lem:max-cliques-dist-by-low-order-sep}, such separations even satisfy that, up to renaming sides, $X \subseteq A$ and $X$ meets $A \setminus B$, and $Y \subseteq B$ and $Y$ meets $B \setminus A$.}
If $\{A,B\}$ has minimum order among the separations of $G$ that separate $X$ and $Y$, then $\{A, B\}$ distinguishes $X$ and $Y$ \defn{efficiently}.
Note that if~$\{A, B\}$ distinguishes~$X$ and~$Y$ (efficiently), then so does its separator~$A \cap B$.
A set~$S$ of separations of $G$ \defn{efficiently distinguishes} a set $\cX$ of vertex-subsets of $G$ if every two elements of $\cX$ are efficiently distinguished by some $s \in S$.

\begin{lemma}\label{lem:max-cliques-dist-by-low-order-sep}
    Let $X$ and $Y$ be two distinct maximal cliques in a locally finite, chordal graph $G$.
    Then $X$ and $Y$ are distinguished by a tight separation of order less than the sizes of $X$ and $Y$.
\end{lemma}
\begin{proof}
    Since $G$ is locally finite, the cliques $X$ and $Y$ are finite.
    Hence, \cref{thm:smallEffSepOfMaxCliques} yields that an efficient $X$--$Y$ distinguishing separation has order less than the size of $X$ and $Y$.
    By~\cref{lem:EffCliqueDistAreTight} and its proof, both the separator of the separation and the separation itself are tight. 
\end{proof}

Given two maximal cliques~$X$ and~$Y$ in a locally finite, chordal graph~$G$, let~$\mathdefn{\beta(X,Y)} = \beta(Y,X)$ be the set of separations efficiently distinguishing $X$ and $Y$; in particular, all separations in~$\beta(X,Y)$ have the same finite order and, by~\cref{lem:max-cliques-dist-by-low-order-sep}, the separations in~$\beta(X,Y)$ have order less than~$|X|$ and $|Y|$.
This implies by~\cite[Example~3.4]{computelocalSeps} that $\beta(X,Y)$ is a bottleneck.
In this paper, we will use the notion of bottlenecks as a black-box and refer the reader to~\cite[Section~3]{computelocalSeps} for the definition.

\begin{lemma} \label{lem:max-clique-dist-are-bottlenecks}
    Let $X$ and $Y$ be two distinct maximal cliques in a locally finite, chordal graph $G$.
    Then the set~$\beta(X, Y)$ is a bottleneck in~$G$.
\end{lemma}
\begin{proof}
    Combine~\cref{lem:max-cliques-dist-by-low-order-sep} with \cite[Example~3.4]{computelocalSeps}.
\end{proof}

In a second step, we find a canonical nested set of separations which meets all the bottlenecks~$\beta(X,Y)$.
For this, we will apply~\cite[Construction~3.6]{computelocalSeps} to the set~$\cB$ of all bottlenecks in the graph, including all those of the form~$\beta(X,Y)$.
We restate this construction here:

\begin{construction}[{\cite[Construction~3.6]{computelocalSeps}}] \label{construction:N(G)}
    Let $G$ be a graph, and let $\cB$ be a set of bottlenecks in $G$.
    For every $k \in \N$, let $\mathdefn{\cB^k(G)}$ be the set of all $k$-bottlenecks in $\cB$ in $G$ and $S^k:=\bigcup\cB^k(G)$.
    Recursively for all $k\in\N$, define nested sets $N^k(G) \subseteq S^k$ of separations, as follows.
    Let $\mathdefn{x^k}\colon S^k\to\N$ assign to each separation $s\in S^k$ the number of separations in $S^k$ that $s$ crosses, assuming for now that this number is finite.
    Write $\mathdefn{N^{<k}(G)}:=\bigcup_{j < k} N^j(G)$.
    For each $\beta \in \cB^k(G)$, let \mathdefn{$N^k(G,\beta)$} consist of the separations $s\in \beta$ which are nested with $N^{<k}(G)$ and among those have minimal~$x^k(s)$.
    Finally, set $\mathdefn{N^k(G)} := \bigcup\,\{\,N^k(G,\beta):\beta\in\cB^k(G)\,\}$.
    We further write $\mathdefn{N(G)}:=\bigcup_{k\in\N} N^k(G)$.
\end{construction}

\cref{construction:N(G)} yields a well-defined nested set of separations, as shown in~\cite{computelocalSeps}; in particular, the function~$x^k: S^k \to \N$ is indeed well-defined.

\begin{theorem}[{\cite[Theorem~3.7]{computelocalSeps}}]\label{thm:BottleneckNestedSet}
    Let $G$ be a locally finite, connected graph. 
    Given any set $\cB$ of bottlenecks in $G$, the sets $N(G)$ and $N^{<k}(G)$ (for $k\in\N$) as in~\cref{construction:N(G)} are well-defined nested sets of tight separations, and $N^k(G,\beta)$ is non-empty for every $k\in\N$ and every $k$-bottleneck $\beta$ in~$G$.
    Moreover, if~$\cB$ is $\Aut(G)$-invariant, then so are~$N(G)$ and $N^{<k}(G)$ for all $k\in\N$.
\end{theorem}

\noindent Here, a set~$\cB$ of bottlenecks is \defn{$\Aut(G)$-invariant} if the action of~$\Aut(G)$ on the bottlenecks of~$G$ (which is inherited from the action of~$\Aut(G)$ on the separations of~$G$) is invariant on~$\cB$.

\begin{corollary}\label{lem:StructureOfN(G)}
    Let $G$ be a locally finite, connected, chordal graph, and let~$\cB$ be the set of all bottlenecks in~$G$.
    Then $N(G)$ is a well-defined, $\Aut(G)$-invariant, nested set of tight separations of $G$ that efficiently distinguishes every two distinct maximal cliques in $G$.
\end{corollary}
\begin{proof}
    The sets~$\beta(X, Y)$ for distinct maximal cliques~$X$ and~$Y$ in~$G$ are bottlenecks by~\cref{lem:max-clique-dist-are-bottlenecks}, and so all $\beta(X,Y)$ are contained in the set $\cB$.
    Moreover, $\cB$ is clearly $\Aut(G)$-invariant by definition.
    By~\cref{thm:BottleneckNestedSet}, the set~$N(G)$ of tight separations obtained from~\cref{construction:N(G)} is well-defined, $\Aut(G)$-invariant and nested, and it in particular meets every~$\beta(X,Y)$.
    Thus, $N(G)$ efficiently distinguishes every two distinct maximal cliques in~$G$.
\end{proof}

The third and final step of our proof of~\cref{thm:canonicalchordalTD} is to show that~$N(G)$ induces the canonical tree-decomposition $\cT(G) = \cT(N(G))$, for which we will apply~\cref{lem:PointFiniteInducesTD}, and that the canonical \td\ $\cT(N(G))$ is indeed into cliques.
We first check that~$N(G)$ satisfies the assumption of~\cref{lem:PointFiniteInducesTD}, that is, $N(G)$ is point-finite.

\begin{lemma}\label{lem:NisPointFinite}
    Let $G$ be a locally finite, connected chordal graph, and let~$N$ be a nested set of tight separations of~$G$ which efficiently distinguishes every two distinct maximal cliques in~$G$.
    Then the separators of separations in~$N$ are cliques, and~$N$ is point-finite.
\end{lemma}
\begin{proof}
    Since all separations in~$N$ are tight, so are their separators, which are hence cliques by~\cref{thm:CharacterisationChordalViaMinimalVertexSeparator}.
    Since~$G$ is locally finite, every vertex of~$G$ is contained in only finitely many cliques, and since~$G$ is connected, there are only finitely many separations with the same finite separator.
    Thus, every vertex of~$G$ is contained in only finitely many separators of separations in $N$, that is, $N$ is point-finite.
\end{proof}

With \cref{lem:NisPointFinite} at hand, we are now ready to prove~\cref{thm:canonicalchordalTD}.
\begin{proof}[Proof of \cref{thm:canonicalchordalTD}]
    Let $N(G)$ be obtained by applying~\cref{construction:N(G)} to the set~$\cB$ of all bottlenecks in~$G$.
    By~\cref{lem:StructureOfN(G)}, $N := N(G)$ is an $\Aut(G)$-invariant nested set of tight separations of~$G$ that efficiently distinguishes every two distinct maximal cliques in~$G$.
    By \cref{lem:NisPointFinite}, $N$ is point-finite, and the separations in $N$ are proper since they are tight.
    So \cref{lem:PointFiniteInducesTD} yields that $\cT \coloneqq \cT(N)$ is a \td\ $(T,\cV)$ such that the map $f \mapsto s^f$ is a bijection from $E(T)$ to $N$.
    Note that, by \cref{lem:PointFiniteInducesTD}, $\cT$ is regular, point-finite and canonical.

    It remains to show that $\cT$ is into cliques.
    All adhesion sets of~$\cT$ are cliques since $N = \{s^f \mid f \in E(T)\}$ and the separators of separations in $N$ are cliques by~\cref{lem:NisPointFinite}.
    Thus, every part $G_t := G[V_t]$ of $\cT$ is connected, since $G$ is connected.
    Moreover, as $G$ is chordal, so are its induced subgraphs $G_t$.

    Suppose for a contradiction that some $V_t$ is not a clique of $G$, i.e.\ there are two non-adjacent vertices $u,v \in V_t$.
    In particular, $V_t \setminus \{u,v\}$ separates $u$ and $v$ in the part $G_t = G[V_t]$ of $\cT$.
    Now let $X$ be an inclusion-minimal $u$--$v$ separator in $G_t$.
    Then~$X$ is non-empty, as the part $G_t$ of $\cT$ is connected, and~$X$ is a clique by~\cref{thm:CharacterisationChordalViaMinimalVertexSeparator}, as the induced subgraph $G_t$ of $G$ is chordal.

    By applying \cref{thm:CliqueAtTightSep} to the (full) components of~$G_t - X$ containing $u$ and $v$, respectively, there are two vertices $u', v' \in V_t \setminus X$ such that $X$ is a $u'$--$v'$ separator in $G_t$ and both $u'$ and $v'$ are complete to $X$.
    So since~$X$ is a clique, there are maximal cliques $X_{u'}$ and $X_{v'}$ of $G$ that contain $X \cup \{u'\}$ and $X \cup \{v'\}$, respectively.
    It is easy to see that, since the adhesion sets of $(T,\cV)$ are cliques, $X$ is not only a $u'$--$v'$ separator in the part $G_t$ but also in $G$.
    In particular, $X$ separates the cliques $X_{u'} \ni u'$ and $X_{v'} \ni v'$ in $G$. 
    Every $X_{u'}$--$X_{v'}$ separator in $G$ must contain their intersection $X_{u'} \cap X_{v'} \supseteq
     X$.
    Hence, the only $X_{u'}$--$X_{v'}$ separator of minimum size in $G$ is $X$ itself.
    Since $N = \{s^f \mid f \in E(T)\}$ efficiently distinguishes every two distinct maximal cliques such as $X_{u'}$ and $X_{v'}$ by assumption, some separation $s^f$ corresponding to an edge $f$ of $T$ separates $X_{u'}$ and $X_{v'}$ efficiently.
    By the above argument, the separator of~$s^f$ is $X$; in particular, $s^f$ separates $u' \in X_{u'} \setminus X$ and $v' \in X_{v'} \setminus X$.
    But $u', v'$ are in the common bag $V_t$ and thus cannot be separated by a separation $s^f$ corresponding to an edge $f$ of $T$, which is a contradiction.
\end{proof}

For experts, we remark that one may alternatively obtain a nested set like $N(G)$ by \cite[Theorem~1.3]{infinitesplinter}, as every maximal clique in a chordal graph yields a 'robust, regular profile' and by invoking \cref{lem:max-cliques-dist-by-low-order-sep} one may prove that each two of them are 'distinguishable'.
A further construction along the lines of~\cite{CanTreesofTDs} is outlined in the introduction.
In contrast, the nested set from \cite[Definition~5.2]{canonicalGD} does not work in general, as the 'tangles' induced by cliques $X$ have in general only order $2/3 \cdot |X|$, which as a consequence may make the two tangles induced by two distinct maximal cliques indistinguishable.
Finally, we can also not work with entanglements~\cite{entanglements} as the order of the separations in clique-distinguishing bottlenecks is not finitely bounded, which would be a necessary assumption for their~\cite[Theorem~4.2]{entanglements}.

\subsection{Proof of \texorpdfstring{\cref{thm:chordalTDmaximal}}{Theorem 3}} \label{sec:TDCoveringsMaxCliques}

\cref{thm:canonicalchordalTD} asserts that every locally finite, chordal graph admits a canonical \td\ into cliques.
The bags of this decomposition can be non-maximal cliques, and there are even finite chordal graphs which do not have a canonical \td\ into maximal cliques (see~\cref{ex:no-maximal-and-canonical}).
In this section, we prove~\cref{thm:chordalTDmaximal} which asserts that if the chordal graph~$\hat G$ we seek to decompose is a normal covering of a locally finite graph~$G$, then we can find a \td\ into maximal cliques, which is still~$\Gamma(p)$-canonical for the group~$\Gamma(p)$ of the deck transformations of the normal covering~$p: \hat G \to G$.

For this, recall that a \defn{covering} of a (loopless) graph $G$ is a surjective homomorphism $p \colon \hat G \to G$ such that $p$ restricts to an isomorphism from the edges incident to any given vertex $\hat v$ of $\hat G$ to the edges incident to its projection $v \coloneqq p_r(\hat v)$. 
A \defn{deck transformation} of a covering $p \colon \hat G \to G$ is an automorphism $\gamma$ of $\hat G$ that commutes with $p$, i.e.\ $p = p \circ \gamma$; we denote the group of deck transformations of a covering $p$ by \defn{$\Gamma(p)$}.
Further, a covering $p \colon \hat G \to G$ is \defn{normal} if the group $\Gamma(p)$ of deck transformations of $p$ acts transitively on the fibers $p^{-1}(v)$ of $p$.
We remark that $r$-local coverings are normal, so the following theorem is applicable to $r$-local coverings and we will detail this application in~\cref{sec:COMPcanonicalGDviaCoverings}.

Let us now restate~\cref{thm:chordalTDmaximal}, the main result of this section:
\chordalTDmaximal*

Our plan to prove~\cref{thm:chordalTDmaximal} is simple:
We will take the canonical \td\ $(T, \cV)$ of~$\hat G$ into cliques as given by~\cref{thm:canonicalchordalTD}, and we then iteratively contract the $\Gamma(p)$-orbits of edges $st$ of the decomposition tree~$T$ which witness that~$(T, \cV)$ is not into maximal cliques, that is, with $V_s \subseteq V_{t}$.
The difficult part of the proof is to show that this construction indeed results in a $\Gamma(p)$-canonical \td\ into its maximal cliques. \\

One ingredient to our proof of~\cref{thm:chordalTDmaximal} is a simple observation about the interaction of cliques in~$\hat G$ with the group~$\Gamma(p)$ of deck transformations of the covering~$p$:

\begin{lemma}\label{lem:DeckTrafoActFreelyOnCliques}
    Let $G$ be a connected graph and $p \colon \hat G \to G$ a covering.
    Then for every clique~$K$ in $\hat G$, the intersection $K \cap \gamma(K)$ is empty whenever $\gamma \in \Gamma(p) \setminus \{ \id_{\hat G}\}$.
\end{lemma}
\begin{proof}
    As $G$ is simple, it has neither loops nor parallel edges. Thus, every two distinct elements of the same fiber of $p$ have distance $>2$ in $\hat G$. In particular, if two cliques in~$\hat G$ meet, then their union cannot contain two distinct vertices of the same fiber of $p$. This immediately yields the statement, since $\Gamma(p)$ acts freely on~$\hat G$.
\end{proof}

As a sidenote, we can deduce from~\cref{thm:canonicalchordalTD,lem:DeckTrafoActFreelyOnCliques} that the group of deck transformations of a normal covering is free if the cover is chordal:

\begin{corollary}\label{cor:rlocchordalhasfreedecktrafo}
    Let $G$ be a locally finite, connected graph and $p \colon \hat G \to G$ a covering.
    If $\hat G$ is chordal, then its group $\Gamma(p)$ of deck transformations is free.
\end{corollary}
\begin{proof}
    Set $\Gamma \coloneqq \Gamma(p)$.
    By \cref{thm:canonicalchordalTD}, $\hat G$ admits a $\Gamma$-canonical \td\ $(T,\cV)$ into cliques.
    \cref{lem:DeckTrafoActFreelyOnCliques} ensures that $\Gamma$ acts freely on the decomposition tree $T$.
    Hence, it follows from \cite[Theorem~4.2.1]{lohgeometric} that $\Gamma$ is free.
\end{proof}

\begin{proof}[Proof of \cref{thm:chordalTDmaximal}]
    Let $(T,\cV)$ be any \td\ of $\hat G$ that is point-finite, regular, $\Gamma(p)$-canonical and into cliques, such as the one given by~\cref{thm:canonicalchordalTD}.

    \begin{claim} \label{cl:NoSubsetsImpliesMaxCliques}
        If there is no edge $f=st$ of $T$ such that $V_s \subseteq V_t$, then $(T,\cV)$ is into the maximal cliques of $G$, i.e. \looseness=-1
        \begin{enumerate}
            \item \label{item:maximalCliqueUniqueNode} for every maximal clique $X$ of $\hat G$, there is a unique node $t(X)$ of $T$ with $X = V_{t(X)}$, and
            \item \label{item:everyVtMaximalClique} every $V_t$ is a maximal clique.
        \end{enumerate}
    \end{claim}
    
    \begin{claimproof}
        \cref{item:maximalCliqueUniqueNode}:
        First, we show that $t(X)$ exists.
        By \cref{lem:CliqueContainedInBag}, there exists a node $t$ of $T$ whose bag $V_t$ contains~$X$.
        As $(T,\cV)$ is into cliques and $X$ is a maximal clique, $V_t$ is indeed equal to $X$.
        
        Secondly, we show that $t(X)$ is unique.
        Suppose for a contradiction that there are two distinct nodes $t,t'$ of $T$ with $V_t, V_{t'} = X$.
        Then \cref{T2} ensures that all vertices $s$ on the $t$--${t'}$ path in $T$ satisfy $X \subseteq V_s$.
        As $(T,\cV)$ is into cliques and $X$ is a maximal clique, we thus have $X = V_s$.
        Hence, every edge $f = ss'$ on the $t$--${t'}$ path in $T$ contradicts the assumption, as $V_{s} = X = V_{s'}$.
        
        \cref{item:everyVtMaximalClique}:
        Suppose for a contradiction that $V_t$ is not a maximal clique.
        Let $X$ be a maximal clique that contains~$V_t$.
        Let $f = ts$ be the first edge on the $t$--$t(X)$ path in $T$.
        Then \cref{T2} ensures that $V_t \subseteq V_s$, which contradicts our assumption.
    \end{claimproof}
    
    We will now iteratively contract the~$\Gamma(p)$-orbits of edges~$st$ of~$T$ with~$V_s \subseteq V_t$, with the aim of obtaining a \td\ satisfying the assumption of~\cref{cl:NoSubsetsImpliesMaxCliques}.
    Note that since $(T,\cV)$ is $\Gamma(p)$-canonical and regular, the deck transformations of $p$ actuniquely on $T$ such that this action commutes with $t \mapsto V_t$,  by \cref{lem:uniqueaction-on-regular-TD}.

    Let us first note that this is a countable procedure in that~$T$ is countable and so is the number of~$\Gamma(p)$-orbits of~$E(T)$.
    Indeed, since $(T,\cV)$ is point-finite, all $T_v$ are finite, and since the connected graph $G$ is locally finite, $G$ is countable.
    Hence, $T = \bigcup_{v \in V(G)} T_v$ is the union of countably many finite graphs, and thus countable as well.

    Now pick an enumeration~$M_1, M_2, \dots$ of the $\Gamma(p)$-orbits of~$E(T)$.
    Set $I_0 \coloneqq \emptyset$, and write~$[i] := \{1, \dots, i\}$ for~$i \in \N$.
    We iteratively construct an inclusion-wise increasing sequence of $I_i \subseteq [i]$ and consider \td s $(T^i,\cV^i)$ of $\hat G$ obtained by contracting $M^i \coloneqq \bigcup_{j \in I_i} M_j$ in $T$, i.e.\ set $T^i \coloneqq T / M^i$, set $V_t^i \coloneqq \bigcup_{s \in C} V_s$ if $t$ is the contraction vertex given by some component $C$ of $(V(T),M^i)$, and else $V^i_t \coloneqq V_t$.
    The~$I_i$ will be constructed such that $(T^i, \cV^i)$ is into cliques, and we have $V^i_s \nsubseteq V^i_t$ and $V^i_s \nsupseteq V^i_t$ for every edge $f = st \in \bigcup_{j \in [i] \setminus I_i} M_j$.
    It is immediate from the definition that the \td s $(T^i, \cV^i)$ are point-finite, regular and $\Gamma(p)$-canonical, and that the $\Gamma(p)$-orbits of $E(T^i)$ are the $M_j$ with $j \in [i] \setminus I_i$ and the $M_j$ with~$j > i$.

    Assume that $I_{i-1}$ is defined for some integer $i \geq 1$.
    Let $f_i = st$ be an edge in $M_i$.
    If $V_s \nsubseteq V_t$ and $V_s \nsupseteq V_t$, then $I_i \coloneqq I_{i-1}$ is as desired.
    Otherwise, we claim that $I_i \coloneqq I_{i-1} \cup \{i\}$ is as desired.
    For this, we make the following observation, which is also the only part of the proof relying on the assumption that~$\hat G$ is a covering of a graph~$G$.
    \begin{claim} \label{cl:GammaOrbitIsMatching}
        The~$\Gamma(p)$-orbit~$M$ of an edge $st$ of $T$ with $V_s \subseteq V_t$ is a matching in~$T$.
    \end{claim}

    \begin{claimproof}
        Suppose for a contradiction that $M$ contains two incident edges $m_1 m_2$ and $m_2 m_3$ of~$T$.
        Then there is a deck transformation~$\gamma \in \Gamma(p)  \setminus \{\id_{\hat G}\}$ that maps $m_1 m_2$ to $m_2 m_3$ (as unoriented edges).
        We cannot have~$\gamma(m_2) = m_2$, since~$\gamma$ would then fix the clique~$V_{m_2}$ contradicting~\cref{lem:DeckTrafoActFreelyOnCliques}.
        So we must have~$\gamma(m_1) = m_2$. But by our assumption on~$st$, this implies either~$V_{m_1} \subseteq V_{m_2} = \gamma(V_{m_1})$ or~$V_{m_1} \supseteq V_{m_2} = \gamma(V_{m_1})$.
        In both cases, the image~$\gamma(V_{m_1})$ of the clique~$V_{m_1}$ meets~$V_{m_1}$, again contradicting~\cref{lem:DeckTrafoActFreelyOnCliques}.
    \end{claimproof}
   
    Now the \td\ $(T^i,\cV^i)$ of $\hat G$ is (canonically isomorphic to) the one obtained from $(T^{i-1}, \cV^{i-1})$ by contracting $M_i$ in $T^{i-1}$.
    By~\cref{cl:GammaOrbitIsMatching}, only independent edges $st$ with, say, $V^{i-1}_s \subseteq V^{i-1}_t$ have been contracted, so $(T^i,\cV^i)$ is still into cliques.

    It remains to consider the limit step: 
    Set $I \coloneqq \bigcup_{i \in \N} I_i$, and let $(T',\cV')$ be the \td\ of $\hat G$ obtained from $(T,\cV)$ by contracting $\bigcup_{j \in I} M_j$ in $T$.
    Again, $(T', \cV')$ is by definition point-finite, regular, $\Gamma(p)$-canonical, and we have $V'_s \nsubseteq V'_t$ and $V'_s \nsupseteq V'_t$ for every edge $st \in E(T')$.
    Finally, every bag $V'_t$ is a clique of $G$, as otherwise there would have been a bag $V^i_{t'} \subseteq V'_t$ which was not a clique as well, contradicting the recursive construction.
    Altogether, $(T', \cV')$ satisfies the assumptions of~\cref{cl:NoSubsetsImpliesMaxCliques}, completing the proof.
\end{proof}

\section{Canonical graph-decompositions via coverings for locally chordal graphs}\label{sec:COMPcanonicalGDviaCoverings}

In this section, we use our results on chordal graphs for the study of locally chordal graphs.
With our applications of~\cref{thm:canonicalchordalTD,thm:chordalTDmaximal}, we contribute to the extension of characterisations of chordal graphs to the locally chordal setting.
Our results in particular give the first known characterisation of a graph class by a property of the canonical graph-decomposition~$\cH_r(G)$ that displays its $r$-global structure, as introduced in~\cite{canonicalGD}:

\Hrintcliques*

\noindent We give all definitions needed to understand and prove~\cref{thm:Hr-int-cliques} in the course of this section.

\subsection{Background: \texorpdfstring{$r$}{r}-local coverings and \texorpdfstring{$r$}{r}-locally chordal graphs} \label{sec:Background:LocalCoverings}

Let $r \in \N$.
A (graph) homomorphism $p \colon \hat G \to G$ is \defn{$r/2$-ball-preserving} if $p$ restricts to an isomorphism from $B_{\hat G}(\hat v, r/2)$ to $B_{G}(v, r/2)$ for every vertex $\hat v$ of $\hat G$ and $v \coloneqq p_r(\hat v)$.
Thus, a surjective homomorphism $p \colon \hat G \to G$ is a covering of a (simple) graph $G$ if and only if $p$ is $1$-ball preserving.

In \cite{canonicalGD} the \defn{$r$-local covering $p_r \colon G_r \to G$} of a connected graph $G$ is introduced as the covering of~$G$ whose characteristic subgroup is the `$r$-local subgroup' of the fundamental group of $G$.
As the formal definition is not relevant to this paper, we refer the reader to \cite[\S 4]{canonicalGD} for details.
In this paper, we will often use that the $r$-local covering preserves the $r/2$-balls.
Indeed, the $r$-local covering is the universal $r/2$-ball-preserving covering of $G$ \cite[Lemma 4.2, 4.3 \& 4.4]{canonicalGD}, i.e.\ for every $r/2$-ball-preserving covering $p \colon \hat G \to G$ there exists a covering $q \colon G_r \to \hat G$ such that $p_r = p \circ q$.
Following this equivalent description of the $r$-local covering, it is also defined for non-connected graphs $G$.

We refer to the graph $G_r$ as the \defn{$r$-local cover} of $G$.
Note that the $0$-, $1$- and $2$-local covers of a (simple) graph $G$ are forests. For a reader familiar with coverings, we remark that the $0$-, $1$- and $2$-local coverings are all indeed the universal covering of $G$.

A graph is \defn{$r$-locally chordal} if its balls of radius $r/2$ are chordal.
The following theorem gives various characterisations of $r$-locally chordal graphs; for our applications, the equivalence between~\cref{basic:i} and~\cref{basic:iii} is the relevant one.

\begin{theorem}[{\cite[Theorem~1]{LocallyChordal}}]\label{BasicCharacterization}
    Let $G$ be a (possibly infinite) graph and let $r \geq 3$ be an integer. The following are equivalent:
    \begin{enumerate}
        \item \label{basic:i} $G$ is $r$-locally chordal.
        \item\label{basic:ii} $G$ is $r$-chordal and wheel-free.
        \item \label{basic:iii} The $r$-local cover $G_r$ of $G$ is chordal.
        \item\label{basic:iv} Every minimal $r$-local separator of $G$ is a clique. 
    \end{enumerate}
\end{theorem}

\subsection{Background: Graph-decompositions and coverings} \label{sec:Background:GraphDecsAndCoverings}

Graph-decompositions, as introduced in~\cite{canonicalGD}, generalize the notion of tree-decompositions by allowing the bags $V_h$ of a decomposition $(H, \cV)$ to be arranged along a general graph $H$ instead of a tree.

Let $G$ be a graph.
A \defn{\gd} of $G$ is a pair $\cH = (H, \cV)$ consisting of a graph $H$ and a family~$\cV$ of vertex-subsets $V_h$ of $G$ indexed by the nodes $h$ of $H$ which satisfies
\begin{enumerate}[label=(H\arabic*)]
    \item \label{axiomH1} $G = \bigcup_{h \in H} G[V_h]$, and
    \item \label{axiomH2} for every vertex $v$ of $G$, the subgraph $H[W_v]$ of $H$ induced on the set $W_v$ of nodes whose bags contain~$v$ is connected.
\end{enumerate}
We refer to $H$ as the \defn{model graph}, and also say that $(H, \cV)$ is a \defn{graph-decomposition over} or \defn{modelled on $H$}.
The vertex-subsets $V_h$ are called the \defn{bags} of $\cH$, and the $W_v$ are called the \defn{co-bags} of $\cH$.
One may flesh out the bags and co-bags by choosing spanning subgraphs $G_h$ of $G[V_h]$ with $G = \bigcup_{h \in H} G_h$ and connected, spanning subgraphs $H_v$ of $H[W_v]$.
Then the $G_h$ are the \defn{parts} and the $H_v$ are the \defn{co-parts} of the decomposition $\cH$.

Given a normal covering $p \colon \hat G \to G$ of a graph~$G$ and a~$\Gamma(p)$-canonical \td\ of~$\hat G$, we can obtain a graph-decomposition of~$G$ by `folding' the \td\ along~$p$, as described by the action of~$\Gamma(p)$:

\begin{construction}[{\cite[Construction~3.8 and Lemma~3.9]{canonicalGD}}]
\label{const:TDFoldingToGD}
    Let $p \colon \hat G \to G$ be any normal covering of a (possibly infinite) connected graph $G$.
    Let $(T, \hat \cV)$ be a $\Gamma$-canonical tree-decomposition of $\hat G$, where $\Gamma \coloneqq \Gamma(p)$ is the group of deck transformations of $p$.
    We define the \gd\ $(H,\cV)$ \defn{obtained from $(T,\hat \cV)$ by folding via $p$} as follows.
    The model graph $H$ is the orbit-graph of $T$ under the action of $\Gamma$, i.e.\ $H = T/
    \Gamma$. 
    Denote the quotient map from $T$ to the orbit-graph $H = T / \Gamma$ also by $p$.
    The bag $V_h$ corresponding to a node $h$ of $H$ is $p(\hat V_t)$ for a vertex (equivalently: every vertex) $t$ in the $\Gamma$-orbit $h$.
    The parts $G_h$ of $(H,\cV)$ are the projections $p(\hat G[\hat V_t])$.
    As co-parts $H_v$ of $(H,\cV)$ we choose the projections $p(T_{\hat v})$ of the co-part $T_{\hat v}$ of $\cT$ corresponding to a (equivalently: every) lift $\hat v$ of $v$ to $\hat G$, where we denote also the quotient map from $T$ to the orbit-graph $H = T / \Gamma$ by $p$.
\end{construction}

\subsection{Coverings and Cliques}

In this section, we study the interactions of an $r$-local covering~$p_r: G_r \to G$ with the (maximal) cliques of~$G_r$ and~$G$.
Some results in this section also appear in~\cite{localGlobalChordal} taking a slightly different perspective (see \cite[Lemma~4.8]{localGlobalChordal} and its proof, and our proof of~\cref{lem:FoldingIntoCliques}~\cref{item:MaxCliquesToMaxCliques}).

\begin{lemma}\label{lem:FoldingIntoCliques}
    Let $G$ be a (possibly infinite) graph and $r \geq 3$ an integer.
    Let $\cT$ be
    a $\Gamma(p_r)$-canonical \td\ of $G_r$, and let $\cH$ be the \gd\ of $G$ obtained from $\cT$ by folding via $p_r$.
    \begin{enumerate}
        \item \label{item:CliquesToCliques} $\cT$ is into cliques if and only if $\cH$ is into cliques.
        \item \label{item:MaxCliquesToMaxCliques} $\cT$ is into maximal cliques if and only if $\cH$ is into maximal cliques.
    \end{enumerate}
\end{lemma}

As tools for the proof of~\cref{lem:FoldingIntoCliques}, we have some easy observations about cliques in covers and a lemma about tree-decompositions into disjoint unions of cliques.
\begin{lemma}[{\cite[Lemma~5.15]{computelocalSeps}}]\label{lem:CliqueAndLocalCover}
    Let $G$ be a (possibly infinite) graph and let $r \geq 3$ an integer.
    \begin{enumerate}
        \item \label{item:CliquesProject} For every clique $\hat X$ of $G_r$, its projection $p_r(\hat X)$ is a clique of $G$, and $p_r$ restricts to a bijection from $\hat X$ to $p_r(\hat X)$.
        \item \label{item:CliquesLift} For every clique $X \subseteq V(G)$ and every lift~$\hat x$ of some vertex~$x \in X$ to~$G_r$, there exists a unique clique $\hat X$ of $G_r$ containing~$\hat x$ such that $p_r$ restricts to a bijection from $\hat X$ to $X$.\footnote{This statement is slightly stronger than what is stated in~\cite[Lemma~5.15]{computelocalSeps}, but follows immediately from its proof.}
        \item \label{item:CliquesMove} $N_{G_r}[\hat X] \cap N_{G_r}[\gamma(\hat X)] = \emptyset$ for every clique $\hat X$ of $G_r$ and every $\gamma \in \Gamma(p_r) \setminus \{\id_{G_r}\}$.
    \end{enumerate}
\end{lemma}

\begin{lemma} \label{lem:CliquesInTreeDecs}
    If~$(T, \cV)$ is a tree-decomposition of a connected graph~$G$ such that each~$G[V_t]$ is a disjoint union of complete graphs, then~$(T, \cV)$ is into cliques, i.e.\,each~$G[V_t]$ is in fact a complete graph.
\end{lemma}
\begin{proof}
    Suppose for a contradiction that there is~$t \in V(T)$ such that~$G[V_t]$ is the disjoint union of at least two complete graphs; in particular, $G[V_t]$ has at least two connected components~$C$ and~$C'$.
    Since~$G$ is connected, there is a path connecting such~$V(C)$ and~$V(C')$ in~$G$, and we choose~$P$ to be a shortest such path.
    Then $P$ meets every clique in~$G$ in at most a vertex or an edge.

    Let~$T'$ be a node-minimal subtree of~$T$ containing~$t$ such that~$V(P) \subseteq \bigcup_{s \in T'} V_s$. 
    Since $P$ meets two distinct components of $G[V_t]$, it follows that $T' \setminus \{t\}$ is non-empty.
    Let~$t' \neq t$ be any leaf of~$T'$.
    Then~$P \cap G[V_{t'}]$ is a disjoint union of subpaths~$P'$ of~$P$, each of length at most~$1$.
    Now an end vertex of~$P'$ is either also an end vertex of~$P$ or it has a neighbour outside of~$V_{t'}$.
    In either case, the end vertex of~$P'$ must be contained in~$V_{t'} \cap V_{t''}$, where $t''$ is the unique neighbour of the leaf~$t'$ of~$T$.
    Thus, the path~$P'$, having length at most~$1$, is also contained in~$G[V_{t''}]$.
    This implies $V(P) \subseteq \bigcup_{s \in T' \setminus \{t'\}} V_s$, contradicting the choice of~$T'$.
\end{proof}

\begin{proof}[Proof of~\cref{lem:FoldingIntoCliques}]
    \cref{item:CliquesToCliques}:
    If the \td\ $\cT$ is into cliques, then \cref{lem:CliqueAndLocalCover}~\cref{item:CliquesProject} yields that the \gd\ $\cH$ is into cliques.
    Conversely, suppose that the \gd\ $\cH$ is into cliques.
    We claim that the \td\ $\cT =: (T, \cV)$ of~$G_r$ has the property that each of its bags is a disjoint union of cliques; \cref{lem:CliquesInTreeDecs} then completes the proof.
    Since~$\cH$ is into cliques, $p_r$ maps each bag~$V_t$ to a clique~$X$ of~$G$. 
    By~\cref{lem:CliqueAndLocalCover}~\cref{item:CliquesLift} and~\cref{item:CliquesMove}, the cliques~$\hat X$ of~$G_r$ for which $p_r$ restricts to a bijection to~$X$ are precisely the vertex sets of the connected components of~$G_r[p_r^{-1}(X)]$. 
    As~$V_t \subseteq p_r^{-1}(X)$, every connected component of~$G_r[V_t]$ is hence complete, as desired.

    \cref{item:MaxCliquesToMaxCliques}:
    Given a maximal clique~$\hat X$ of~$G_r$, \cref{lem:CliqueAndLocalCover}~\cref{item:CliquesProject} ensures that $X \coloneqq p_r(\hat X)$ is a clique of $G$.
    Let $Y$ be a maximal clique of $G$ containing $X$.
    By \cref{lem:CliqueAndLocalCover}~\cref{item:CliquesLift}, there is a lift $\hat Y$ of $Y$ to $G_r$.
    By potentially shifting $\hat Y$ with a deck transformation of $p_r$, we may assume that $\hat Y$ meets $\hat X$.
    Since distinct lifts of the clique $X$ are disjoint by \cref{lem:CliqueAndLocalCover}~\cref{item:CliquesMove}, $\hat Y$ must contain the whole lift $\hat X$ of $X$.
    In particular, since $\hat X$ is a maximal clique of $\hat G$, $\hat Y = \hat X$, and thus $X = Y$ is a maximal clique of $G$.
    Similarly, one proves that every clique of $G_r$ that projects to a maximal clique of $G$ is also maximal.
    So since distinct cliques of~$G_r$ that project to the same clique of~$G$ are disjoint by \cref{lem:CliqueAndLocalCover}~\cref{item:CliquesMove}, $\cT$ is into maximal cliques if and only if $\cH$ is into maximal cliques.
\end{proof}

We remark that both \cref{lem:FoldingIntoCliques,lem:CliqueAndLocalCover} and their proofs hold more generally for all coverings~$p: \hat G \to G$ which are $3/2$-ball preserving.

\subsection{Canonical graph-decompositions into cliques} \label{sec:CanGraphDecIntoCliques}

Recall the theorem giving the canonical graph-decomposition~$\cH_r(G)$:

\CanGraphDec*

To understand~\cref{thm:CanGraphDec}, let us detail the construction of the graph-decomposition~$\cH_r(G)$ given by~\cref{thm:CanGraphDec}.
We will not give the original construction from~\cite{canonicalGD}, but the most recent one from~\cite{computelocalSeps}, which builds on the notion of bottlenecks instead of tangles\footnote{As described in~\cref{sec:ConstructionN(G)}, sets of separations efficiently distinguishing two tangles form a bottleneck. Thus, the approach building on bottlenecks captures more than a tree of tangles. We refer the reader to~\cite{computelocalSeps}, in particular \S1.4, for a detailed comparison of the two constructions.}.

Let~$p_r : G_r \to G$ be the $r$-local covering of a locally finite, connected, $r$-locally chordal graph~$G$. 
By~\cref{BasicCharacterization}, $G_r$ is a locally finite, connected and chordal graph.
The `more precisely'-part of \cref{thm:canonicalchordalTD} then implies that $\cT(G_r) = \cT(N(G_r))$ is a canonical, regular, point-finite \td\ of~$G_r$.
We can thus invoke~\cref{const:TDFoldingToGD} to obtain the canonical \gd\ $\cH_r(G)$ of~$G$.
For a thorough discussion on how~$\cH_r(G)$ displays the~$r$-global structure of~$G$ and why it is the unique canonical such \gd\ of~$G$, we refer the reader to~\cite{canonicalGD}.

\begin{proof}[Proof of \cref{thm:Hr-int-cliques}]
    First, assume that $G_r$ is chordal.
    By \cref{thm:canonicalchordalTD}, $\cT(G_r)$ is into cliques, and by \cref{lem:FoldingIntoCliques}~\cref{item:CliquesToCliques}, $\cH_r(G)$ is thus also into cliques.

    Conversely, assume that $\cH_r(G)$ is into cliques.
    Then \cref{lem:FoldingIntoCliques}~\cref{item:CliquesToCliques} yields that the \td\ $\cT(G_r)$ of~$G$ from which $\cH_r(G)$ is obtained by folding via~$p_r$ is into cliques.
    Hence, this \td\ witnesses by \cref{prop:subtreeintersec-is-chordal} that the $r$-local cover $G_r$ of $G$ is chordal.
\end{proof}

We remark that the proof of \cref{thm:Hr-int-cliques} carries over to locally finite graphs, except for the existence of~$\cH_r(G)$, which is not guaranteed by \cref{thm:CanGraphDec} for locally finite graphs~$G$.
For the forward implication of \cref{thm:Hr-int-cliques}, if $G$ is a locally finite, connected, $r$-locally chordal graph, then $\cH_r(G)$ exists, since $\cT(G_r)$ then is a canonical \td\ of the locally finite, connected, chordal graph $G_r$ by~\cref{thm:canonicalchordalTD}.
The backward implication of \cref{thm:Hr-int-cliques} only remains valid if $\cH_r(G)$ exists for the locally finite graph $G$.
In particular, \cref{thm:Hr-int-cliques} holds for quasi-transitive, locally finite graphs $G$, where $\cH_r(G)$ exists by \cite[Theorem~5.5]{canonicalGD}.

\subsection{Graph-Decompositions into maximal cliques} \label{sec:GDsIntoMaximalCliques}

In this section, we discuss~\cref{thm:chordalTDmaximal} in the context of~$r$-locally chordal graphs.
Its main application lies in the proof of a further characterization of~$r$-locally chordal graphs; for this, a \gd $(H, \cV)$ of a graph~$G$ is \defn{$r$-acyclic} for an integer~$r \ge 0$ if for every set~$X$ of at most~$r$ vertices of~$G$, the union of the respective co-parts~$H_x$ with~$x \in X$ is acyclic.

\begin{theorem}{\cite[Theorem 4]{localGlobalChordal}} \label{mainresult:racycliccliquegraphs-iff-r-locally-chrodal-etc}
    Let $G$ be a locally finite graph and $r \geq 3$ an integer.
    Then the following are equivalent:
    \begin{enumerate}
        \item \label{locallyChordal:i} $G$ is $r$-locally chordal.
        \item \label{locallyChordal:ii} $G$ admits an $r$-acyclic \gd\ into maximal cliques.
        \item \label{locallyChordal:iii} $G$ admits a canonical $r$-acyclic \gd\ into cliques.
        \item \label{locallyChordal:iv} $G$ admits an $r$-acyclic \gd\ into cliques.
    \end{enumerate}
\end{theorem}

For the proof of~\cref{mainresult:racycliccliquegraphs-iff-r-locally-chrodal-etc}, we refer the reader to~\cite{localGlobalChordal}.
\cref{thm:canonicalchordalTD} features in the proof of `\cref{locallyChordal:i} implies \cref{locallyChordal:iii}' in that it gives the input to~\cref{const:TDFoldingToGD}, which then yields the desired canonical graph-decompositions into cliques by~\cref{lem:FoldingIntoCliques}~\cref{item:CliquesToCliques} (as in the proof of~\cref{thm:Hr-int-cliques}).
Similarly, \cref{thm:chordalTDmaximal} features in the proof of~`\cref{locallyChordal:i} implies \cref{locallyChordal:ii}' in that it gives the input to~\cref{const:TDFoldingToGD}, which then yields the desired graph-decompositions into maximal cliques by~\cref{lem:FoldingIntoCliques}~\cref{item:MaxCliquesToMaxCliques}.
As shown in~\cite{localGlobalChordal}, both these \gd s turn out to be also $r$-acyclic.

\section{Counterexamples to possible strengthenings of \texorpdfstring{\cref{thm:canonicalchordalTD}}{Theorem 2}} \label{sec:CounterEx}

\cref{thm:canonicalchordalTD} asserts that every locally finite, chordal graph admits a canonical \td\ into cliques.
In this section, we provide counterexamples to two possible strengthenings of~\cref{thm:canonicalchordalTD}.
First, \cref{ex:no-maximal-and-canonical} shows that we cannot obtain a canonical \td s into maximal cliques.
Secondly, \cref{ex:CountableNoCanonical} demonstrates that a (countable) graph need not admit a canonical \td\ into cliques, even if it admits non-canonical one.

\subsection{Canonical \td s into maximal cliques}

Recall that Halin's \td\ in \cref{thm:NoInifniteCliqueChordalIsTDIntoCliques} are not only into cliques but even into maximal cliques.
Formally, a \td\ $(T,\cV)$ of a graph $G$ is \defn{into its maximal cliques} if the map $t \mapsto V_t$ is a bijection from $V(T)$ to the set of maximal cliques of $G$.
It is natural to ask if \cref{thm:canonicalchordalTD} extends to canonical \td s into maximal cliques.
But there are even finite (connected) chordal graphs that do not admit a canonical \td\ into their maximal cliques:

\begin{example}\label{ex:no-maximal-and-canonical}
    The stars $K_{1,t}$ with $t \geq 3$ leaves are chordal graphs which do not admit a canonical \td\ into its maximal cliques.
\end{example}

\begin{proof}
    For every \td\ $(T, \cV)$ of $K_{1, t}$ into its maximal cliques, the bags $\cV$ are exactly the set of edges of $K_{1, t}$. For any pair of edges $e, f$ of $K_{1, t}$, there is an automorphism of $G$ mapping $e$ to $f$, so if $(T, \cV)$ is canonical, then decomposition graph $T$ must be vertex-transitive. But a finite, vertex-transitive tree has at most two vertices, since every finite tree with at least three vertices contains a vertex of degree one and a vertex of degree greater than one. Therefore, if $t \geq 3$, then the graph $K_{1, t}$ does not admit a canonical \td\ into its maximal cliques. 
\end{proof}

As a sanity check, we remark that the stars $K_{1,t}$ with $t \ge 3$ from~\cref{ex:no-maximal-and-canonical}, however, admit a canonical tree-decomposition into cliques:
take a $K_{1,t}$ as the decomposition tree and associate with each of its leaves the incident edge as its bag and for the center vertex, choose the bag as the singleton center vertex, which is a non-maximal clique.

Further, we remark that~\cref{ex:no-maximal-and-canonical} is still a counterexample if we drop the requirement of bijectivity from the definition of `into its maximal cliques', that is, if we allow the \td\ to have the same maximal clique as different bags.
Indeed, the nodes of the decomposition tree~$T$ with the same bag form a subtree of~$T$, and the edge-transitivity of~$K_{1,t}$ then translates into transitivity of all those subtrees of~$T$.
So after contracting each such subtree to a single node, we can conclude with the same argument as above.

In light of Diestel's characterisation of countable graphs that admit a \td\ into cliques via forbidden simplicial minors~\cite{CharacterisationOfCountableChordalGraphsAdmitTD}, it might be a natural goal to characterise the countable, chordal graphs that admit canonical \td s into cliques or characterise the countable, chordal graphs that admit \td s into its maximal cliques.

\subsection{Canonical tree-decomposition into cliques}

Does every graph admitting a \td\ into cliques also admit one that is canonical?
We answer this question in the negative, even for countable graphs.
Our example is inspired by \cite[Figure 1]{CanTreesofTDs}, who attribute their example to~\cite{dunwoody2015vertex}.

\begin{example} \label{ex:CountableNoCanonical}
    The graph $G$ depicted in \cref{fig:CEcanonicalTD} is a countable, chordal graph that admits a \td\ into cliques but not a canonical such \td .
\end{example}

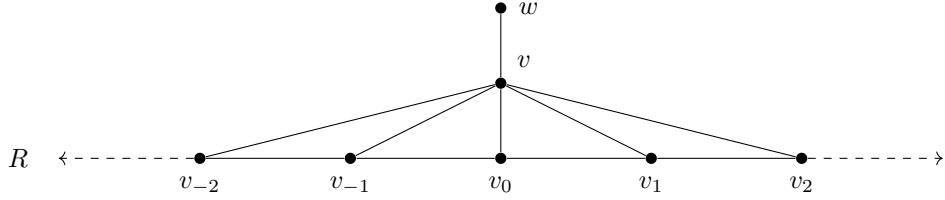
\begin{figure}
    \centering

    \begin{tikzpicture}
        \node[opacity=0] (left) at (0,0) {};
        \node[left=1pt of left] {$R$};
        \node[circle, fill=black, inner sep=1.5pt] (vminus2) at (2,0) {};
        \node[below=1pt of vminus2] {$v_{-2}$};
        \node[circle, fill=black, inner sep=1.5pt] (vminus1) at (4,0) {};
        \node[below=1pt of vminus1] {$v_{-1}$};
        \node[circle, fill=black, inner sep=1.5pt] (vzero) at (6,0) {};
        \node[below=1pt of vzero] {$v_{0}$};
        \node[circle, fill=black, inner sep=1.5pt] (vone) at (8,0) {};
        \node[below=1pt of vone] {$v_{1}$};
        \node[circle, fill=black, inner sep=1.5pt] (vtwo) at (10,0) {};
        \node[below=1pt of vtwo] {$v_{2}$};
        \node[opacity=0] (right) at (12,0) {}; 
        
        \node[circle, fill=black, inner sep=1.5pt] (v) at (6,1) {};
        \node[above right=1pt and 1pt of v] {$v$};
        \node[circle, fill=black, inner sep=1.5pt] (w) at (6,2) {};
        \node[right=1pt of w] {$w$};
        
        \draw[dashed,->] (vminus2) -- (left);
        \draw (vminus2) -- (vminus1);
        \draw (vminus1) -- (vzero);
        \draw (vzero) -- (vone);
        \draw (vone) -- (vtwo);
        \draw[dashed,->] (vtwo) -- (right);
        
        \draw (v) -- (vminus2);
        \draw (v) -- (vminus1);
        \draw (v) -- (vzero);
        \draw (v) -- (vone);
        \draw (v) -- (vtwo);
        
        \draw (w) -- (v);
    \end{tikzpicture}
    \caption{The graph~$G$ for \cref{ex:CountableNoCanonical} consists of a double-ray $R = \dots e_{-1}v_{-1}e_{0}v_0e_1v_1 \dots$ with an apex vertex $v$ and a further edge $e = vw$ attached to $v$.}
    \label{fig:CEcanonicalTD}
\end{figure}

\begin{proof}
    For a \td\ into cliques, we take the double-ray $R$ together with the edge $v_0 v$ as the decomposition tree.
    We let $V_v = \{v, w\}$ and $V_{v_i} = \{v, v_i, v_{i+1}\}$.
    It is easy to check that this is indeed a \td , and it is by the definition of~$G$ one into cliques.

    Suppose now that $(T,\cV)$ is a \td\ of $G$ into cliques which is even canonical.
    The two triangles $v_{-1} v_0 v$ and $v_0 v_1 v$ are both contained in bags of this \td , and these bags have to be distinct because $(T, \cV)$ is into cliques.
    Further, these two bags intersect in $\{v_0, v\}$, hence there is an edge $e_0 \in T$ with precisely $\{v_0, v\}$ as its adhesion set.
    Let $\phi$ be the (unique) automorphism of $G$ which translates the ray $R$ by one, i.e.\ $\phi(e_{i})=e_{i+1}$ and $\phi(v_i) = \phi (v_{i+1})$ for every $i \in \Z$, while $\phi(v) = v$ and $\phi(w) = w$.
    Repeated application of $\phi$ (and $\phi^{-1}$) implies together with the canonicity of $(T, \cV)$ that $e_0$ yields the existence of edges $e_i$ for all $i$ with adhesion set $\{v_i, v\}$.
    The separation induced by $e_i$ contains the edge $\{v, w\}$ on one side, and since~$(T, \cV)$ is canonical, all separations induced by some~$e_i$ must agree on this side as the repeated application of $\phi$ (and $\phi^{-1}$) shows.
    But since $(T, \cV)$ is a \td, this is not possible.
\end{proof}

\subsection{Tree-decompositions into maximal cliques}

Does every graph admitting a \td\ into cliques also admit one  into its maximal cliques?
We answer this question in the negative, even for countable graphs.
We will deduce this from the fact that a certain uncountable chordal graph~$G^*$ does not admit a \td\ into cliques, as Diestel showed in~\cite[Paragraphs after Theorem 1.4]{CounterexampleUncountableChordalAdmitTD}.
A tree~$T$ at a vertex $r$ naturally induces a partial order, the \defn{tree-order $\leq_T$}, $u \leq_T v$ on its vertices by $u$ lies on the unique $r$--$v$ path $rTv$ in $T$. 

\begin{example} \label{ex:CountableNoMaxCliques}
    Let $T_2$ be the infinite binary tree with root $r$, i.e.\ the root $r$ has degree $2$ and every other vertex of $T_2$ has degree $3$.
    We obtain a chordal graph $G$ from $T_2$ by adding for each two $\leq_{T_2}$-comparable vertices $u \leq v$ of $T_2$ the edge $uv$ to $G$.
    Then $G$ admits a (canonical) \td\ into cliques, but it does not admit one into its maximal cliques.
\end{example}

\begin{proof}
    For a \td\ of $G$ into cliques, we choose $T = T_2$ as the decomposition graph and for each $t \in T$, we let $V_t = \{v \in G : v \le_{T_2} t\}$, the vertex set of the unique $r$--$t$ path in $T_2$.
    It can be easily checked that this is indeed a \td , and its bags are cliques by the definition of $G$ from $T_2$.

    Suppose for a contradiction that $G$ admits a \td , $(T,\cV)$ into maximal cliques.
    Let $G^*$ be the graph obtained from $G$ by adding to each ray $R$ in the tree $T_2$ starting from $r$ a vertex $v_R$ and an edge~$uv_R$ to each vertex $u$ of $R$.
    Note that the maximal cliques of $G$ are precisely the vertex sets of the rays in~$T_2$ starting from $r$.
    Thus, for each such ray $R$, we add $v_R$ to the bag $V_{t_R}$ to obtain a \td\ $(T,\cV')$ of $G^*$ which is not only into cliques but even into maximal cliques.
    This contradicts that $G^*$ does not admit a \td\ into cliques, as Diestel showed in~\cite[Paragraphs after Theorem 1.4]{CounterexampleUncountableChordalAdmitTD}.
\end{proof}

\printbibliography

\end{document}